\documentclass[12pt]{amsart}
\title[Bounded geometry wandering domains]{No bounded geometry wandering domains for sufficiently regular automorphisms}

\author[S. Merenkov]{Sergei Merenkov}
\address{Department of Mathematics, City College of New York and CUNY Graduate Center, New York, NY 10031, USA}
\email{smerenkov@ccny.cuny.edu}
\thanks{Supported by NSF grant DMS-1800180.}

\usepackage{amsmath, amssymb, amsthm,latexsym}
\usepackage{graphicx}
\newcommand\C{{\mathbb C}}

\newcommand\N{{\mathbb N}}
\newcommand\D{{\mathbb D}}

\newcommand\Z{{\mathbb Z}}
\newcommand\R{{\mathbb R}}
\newcommand\T{{\mathbb T}}

\newcommand\hC{{\hat{\mathbb C}}}

\newcommand\dee{\partial}

\newcommand\id{\operatorname{id}}

\renewcommand\:{\colon}
\newcommand\sub {\subseteq}

\newcommand\Ga{\Gamma}

\newtheorem{theorem}{Theorem}[section]

\newtheorem{lemma}[theorem]{Lemma}

\newtheorem{proposition}[theorem]{Proposition}
\newtheorem{corollary}[theorem]{Corollary}
\newtheorem{question}[theorem]{Question}

\theoremstyle{definition}

\usepackage{hyperref}
\hypersetup{
    colorlinks=true,
    linkcolor=blue,
    filecolor=magenta,      
    urlcolor=cyan,
}

\begin{document}

%\email{}

\abstract{
A question whether sufficiently regular manifold automorphisms may have wandering domains with controlled geometry is answered in the negative for quasiconformal or smooth homeomorphisms of $n$-tori, $n\ge2$, and 
hyperbolic surfaces. Besides control on geometry of wandering domains, the assumptions are either analytic, e.g., minimal sets having measure zero or supporting invariant conformal structures, or geometric, such as uniform relative separation of wandering domains.
} 
\endabstract

\maketitle

\setcounter{tocdepth}{1}
\tableofcontents

\section{Introduction}\label{s:Intro}

\noindent
A classical result of A.~Denjoy~\cite{De32} states  that the action of an orientation preserving $\mathcal C^{1+bv}$-diffeomorphism $f$ of the circle (i.e., $\mathcal C^1$-diffeomorphism with $Df$ of bounded variation) without periodic points is {minimal}, i.e., points have dense orbits on the circle. In particular, any such  $\mathcal C^2$-diffeomorphism is topologically conjugate to an irrational rotation. The condition $\mathcal C^{2}$ is H\"older-sharp in the sense that for each $\alpha\in[0,1)$ there are $\mathcal C^{1+\alpha}$ counterexamples~\cite{He79}.

Naturally, it is interesting to find out whether similar phenomena occur in higher dimensions.
A particular question of whether a diffeomorphism of 2-torus that is semiconjugate (but not conjugate) to a mi\-ni\-mal translation can have uniform conformal geometry originates from work of A.~Norton and D.~Sullivan~\cite{NS96}. 
%More specifically, can wandering domains have uniformly bounded geometry?
%A related result by C.~Bonatti, J.-M.~Gambaudo, J.-M.~Lion, and C.~Tresser, contained in~\cite{BGLT94}, asserts that certain sufficiently smooth infinitely renormalizable diffeomorphisms of the unit disk cannot have wandering domains with bounded geometry in a suitable sense. 

\subsection{Denjoy type automorphisms}
In the terminology of \cite{NS96},  extended to higher dimensions, an orientation preserving homeomorphism $f$ of $n$-torus $\T^n,\ n\geq2$, has \emph{Denjoy type} if there exists a continuous map $h$ of $\T^n$ to itself, such that 

\noindent
--$h$ is homotopic to the identity,  

\noindent
--the set $V_h$ of \emph{non-trivial} values of $h$ (i.e., elements $p\in\T^n$ with $\#\{h^{-1}(p)\}>1$) is non-empty and countable, 

\noindent
--there is a \emph{minimal} translation $R$ of $\T^n$ 
(i.e., every point has a dense orbit under the action of $R$) 
such that $f$ and $R$ are semiconjugate, i.e.,
$$
h\circ f=R\circ h,
$$  

\noindent
--$f$ is not conjugate to a translation. 

Since $R$ is a translation and $h$ is homotopic to the identity, the functional equation above implies that the map $f$ is isotopic to the identity.
The \emph{minimal set} of $f$ is 
$$
\Lambda=\T^n\setminus\cup_{p\in V_h}\left\{(h^{-1}(p))^o\right\}, 
$$ 
where $X^o$ denotes the interior of a set $X$.
From this definition we conclude that if the interior of $h^{-1}(p)$ is non-empty for at least one $p$, then $\Lambda$ is a closed nowhere dense subset of $\T^n$ that is  \emph{completely  
invariant}, i.e., forward and backward invariant, under $f$. Finally, each non-empty $D_p=(h^{-1}(p))^o$ is a \emph{wandering domain} of $f$, i.e., $f^k(D_p)=f^l(D_p),\ k,l\in\Z$, implies $k=l$. Here and in what follows $f^k$ denotes the $k$'th iterate of $f$.  

%If $\mathcal M$ is a Riemannian $n$-manifold and $\alpha\geq 0$, we denote by $\mathcal C^{\alpha}(\mathcal M)$ the class of continuously differentiable maps $f$ of $\mathcal M$ up to order $[\alpha]$ with $D^{[\alpha]}f$ being $\{\alpha\}$-H\"older in the case $\{\alpha\}>0$. Here, $[\alpha]$ and $\{\alpha\}$ denote the integer and fractional parts of $\alpha$, respectively. 
%We denote $\mathcal C^{k+Z},\ k\in\N\cup\{0\}$, the class of $f\in \mathcal C^k(\mathcal M)$ with \emph{Zygmund} $D^kf$, i.e.,
%$$
%\left| D^kf(x+h)+D^kf(x-h)-2D^kf(x)\right|\leq C |h|,
%$$
%for some $C>0$ and all $x,h\in\R^n$. 
%The condition $f\in \mathcal C^{k+Z}$ is weaker than $\mathcal C^{k+Lip}$, i.e., $f\in\mathcal C^k$ and $D^kf$ is Lipschitz, but stronger than $\mathcal C^{k+H\ddot{o}l}$, i.e., $f\in\mathcal C^k$ and $D^kf$ is $\alpha$-H\"older for some $\alpha\in(0,1)$.
%By ${\rm Diff}^{\alpha}(\mathcal M)$ 
%%or ${\rm Diff}^{k+Z}(\mathcal M)$ 
%we denote the space of $\mathcal C^{\alpha}$-diffeomorphisms of $\mathcal M$.
%%, or, respectively, $\mathcal C^{k+Z}$-

%In~\cite{NS96}, A.~Norton and D.~Sullivan proved the following theorem.
\begin{theorem}\cite[Theorem~2]{NS96}\label{T:NS}
Let a diffeomorphism $f\in{\rm Diff}^1(\T^2)$ have Denjoy type, and suppose $f$ preserves a $\mathcal C^{1+Z}$ conformal structure on $\Lambda$. Then $f$ cannot be in ${\rm Diff}^{2+Z}(\T^2)$.
\end{theorem}
Here, superscript $k+Z,\ k=1,2$, means that the $k$th derivative satisfies the Zygmund condition.
%Also, if $\mathcal M$ is a smooth real $n$-dimensional manifold and $X$ is a measurable subset of $\mathcal M$, a \emph{conformal structure} on $X$ is a choice $c(p)$ of a similarity class on the tangent space $T_p\mathcal M$ of $\mathcal M$ for every $p\in X$.
%Recall that a \emph{similarity class} on a  $n$-dimensional real vector space $V$ is an equivalence class of positive definite inner products on $V$, up to scale. The space $C(V)$ of similarity classes of $V$ can be identified with $SL(n,\R)/SO(n,\R)$. It is well known that in dimension 2 the space $C(V)$ is naturally isometric
%to the hyperbolic plane with respect to the Poincar\'e metric.
%A \emph{measurable}, $\mathcal C^\alpha$, or $\mathcal C^{k+Z}$ conformal structure is one that comes from a measurable, $\mathcal C^\alpha$, or $\mathcal C^{k+Z}$ Riemannian metric, respectively.  A measurable conformal structure on $X$ is  \emph{bounded} if it is essentially bounded, i.e., there exists a positive constant $K$ such that
%$$
%\|\log c(p)\|\le K
%$$
%for almost every $p\in X$, where $\|\cdot\|$ is the Hilbert--Schmidt norm. In dimension 2  measurable bounded conformal  structures are those that are bounded in the Poincar\'e metric and they correspond to Beltrami forms defined in Section~\ref{S:Prelim}.
%A corollary to Theorem~\ref{T:NS} is the following.
\begin{corollary}~\cite[Corollary~1]{NS96}\label{C:NS96}
No Denjoy type $f\in{\rm Diff}^3(\T^2)$ exists such that the wandering domains of $f$ are non-degenerate  geometric disks. 
\end{corollary}
This corollary was generalized by A.~Navas~\cite{Na18} to $\mathcal C^{n+1}$-dif\-fe\-o\-mor\-phisms of $\T^n,\ n\ge2$, and by the author~\cite{Me19} to $\mathcal C^{1}$-diffeomorphisms of $\T^n,\ n\ge2$.

A natural relaxation of the geometric disks assumption as in Corollary~\ref{C:NS96} is requirement of bounded geometry.
A collection of domains $\{D_i\}_{i\in I}$ in a metric space
is said to have \emph{bounded geometry} if  there exist $C\geq1, p_i\in D_i$, and $0<r_i\leq R_i,\ i\in I$, with
\begin{equation}\label{E:Bddgeom}
B(p_i, r_i)\subseteq D_i\subseteq B(p_i, R_i),\quad {\rm and}\quad {R_i}/{r_i}\leq C, 
\end{equation}
where $B(p,r)$ denotes the ball centered at $p$ of radius $r$. 
The following is the first main result; see Section~\ref{S:Prelim} for a definition of  quasiconformality.

\begin{theorem}\label{T:tori}
If $f$ is a Denjoy type homeomorphism of $\T^n, n\geq 2$, whose minimal set $\Lambda$  has measure zero, then either $f$ cannot be quasiconformal or wandering domains of $f$ do not have bounded geometry.
\end{theorem}
%Otherwise, $\{D_i\}_{i\in I}$ is said to have \emph{non-bounded geometry}. 

Theorem~\ref{T:NS} also motivated the following questions.
\begin{question}[A.~Norton, D.~Sullivan~\cite{NS96}]\label{Q:NS}
Can $f\in {\rm Diff}^{2+Z}(\T^2)$ have Denjoy type? 
%\end{question}
%%Theorem~\ref{T:NS} implies that no such map can preserve a $\mathcal C^{1+Z}$ con\-for\-mal structure on $\Lambda$. A natural follow up question is the following.
%\begin{question}[A.~Norton, D.~Sullivan~\cite{NS96}]\label{Q:NS2}
Can Denjoy type $f\in {\rm Diff}^{2+Z}(\T^2)$ preserve a measurable bounded conformal structure on $\Lambda$?
\end{question}

In higher dimensions we allow $\Lambda$ to have positive measure and give a partial answer to the second question.  We say that a set $X$ in a metric space is \emph{locally rectifiably connected} if for each $x\in X$ and a neighborhood $V$ of $x$  there is a neighborhood $U\subseteq V$ of $x$ such that any two points of $U\cap X$ can be joined by a rectifiable path in $U\cap X$.
%Also, the \emph{standard conformal structure} on $\T^n$ is the one that comes from the Euclidean metric of the covering space $\R^n$ of $\T^n$. 
\begin{theorem}\label{T:tori2}
Let $n\ge3$,  and $f$ be a Denjoy type homeomorphism of $\T^n$ such that the minimal set $\Lambda$ of $f$ is not all of $\T^n$, is connected and locally rectifiably connected. Assume further that the Lebesgue density points of $\Lambda$ 
%that are not on the boundary of any wandering domain 
are dense in $\Lambda$. 
Then $f$ cannot be in ${\rm Diff}^3(\T^n)$ and preserve the standard conformal structure on $\Lambda$. 
\end{theorem}

\subsection{Surface automorphisms}
More generally, let $(\mathcal M,g)$ be a \emph{closed}, i.e., compact and with no boundary, Riemannian $n$-manifold. 
%A \emph{domain} $D$ in $\mathcal M$ is an open connected set with closure $\overline D$ being a closed topological ball. 
%that is contractible in $\mathcal M$. 
Following terminology of F.~Kwakkel and V.~Mar\-ko\-vic~\cite{KM10}, we say that a homeomorphism $f$  of $\mathcal M$ \emph{permutes} a dense collection of domains if there exists a connected, completely invariant, and nowhere dense  compact subset $\Lambda\subset \mathcal M$ such that the collection $\{D_i\}_{i\in I}$ of connected components  of the complement $\mathcal M\setminus \Lambda$ has the following properties: for each $i\in I$, the domain $D_i$ is the interior of its closure $\overline{D_i}$,  the closures $\overline{D_i},\ i\in I$, are pairwise disjoint, and for each $i\in I,k\in\Z, \ f^k(D_i)\cap D_i\neq\emptyset$ implies $k=0$. Abusing terminology, we refer to $\Lambda$ as a \emph{minimal set} of $f$ and to the domains $D_i,\ i\in I$, as \emph{wandering domains}, as in the case of Denjoy type homeomorphisms. Note that neither do we assume that the orbit of a single $\overline{D_i}$ is dense in $\Lambda$, nor that there is a unique orbit of wandering domains under $f$. 

In the spirit of lowering regularity, in~\cite{KM10}, F.~Kwakkel and V.~Mar\-ko\-vic posed the following question.
\begin{question}[F.~Kwakkel, V.~Markovic~\cite{KM10}]\label{Q:PosEntr}
If $\mathcal M$ is a closed Riemann surface
equipped with the canonical metric induced from the standard conformal metric of the universal cover, do  there exist diffeomorphisms $f\in{\rm Diff}^1(\mathcal M)$ with positive topological entropy that permute a dense collection of domains with bounded geometry?
\end{question}

F.~Kwakkel and V.~Markovic proved that the answer to Question~\ref{Q:PosEntr} is negative under a stronger assumption that $f\in{\rm Diff}^{1+\alpha}(\mathcal M)$, where $\alpha>0$; \cite[Theorem~A]{KM10}.  
%They did it by controlling the topological entropy by the dilatation of $f$.
%Theorem~\ref{T:C1Diffeos} gives a negative answer to Question~\ref{Q:PosEntr} in the case of Denjoy type $\mathcal C^1$-diffeomorphisms of 2-torus $\T^2$ such that the permuted domains are geometric disks. Also,  
Theorem~\ref{T:tori} answers Question~\ref{Q:PosEntr} in the negative for  Denjoy type quasiconformal homeomorphisms of $\T^2$, %whose wandering domains have bounded geometry,
assuming that minimal sets have measure zero. 
Our next result states that, under additional geometric conditions on the wandering domains, rather than the entropy condition on $f$, the answer to Question~\ref{Q:PosEntr} is also negative; see Section~\ref{S:Prelim} for the uniform quasidisks and uniform relative separation conditions.

\begin{theorem}\label{T:surf}
Let $\mathcal M$ be a closed Riemann surface that is either hyperbolic  or flat  2-torus $\T^2$. Then there is no orientation preserving $f\in{\rm Diff}^1(\mathcal M)$ that permutes a dense collection of uniform quasidisks that are uniformly relatively separated.
\end{theorem}
Note that here we do not assume that the minimal set $\Lambda$ of $f$ has measure zero. 
Theorem~\ref{T:surf} is false in the case when $\mathcal M=\hC$, as can be seen by taking $f(z)=\lambda z,\ \lambda>1$. One needs to ``fill in'' the fundamental annulus $\{z\: 1\le |z|<\lambda\}$ by uniformly relatively separated disks 
%that are also uniformly relatively separated from the two boundary circles of the annulus, 
and spread them out using the dynamics of $f$. The simple details are left to the reader. We require the uniform quasidisks and uniformly relatively separated conditions in Theorem~\ref{T:surf} since we use M.~Bonk's round quasiconformal  uniformization result from~\cite{Bo11}. The assumption that the collection of permuted domains consists of uniform quasidisks is stronger than these domains having bounded geometry, but is in the same spirit nevertheless. The necessity of the assumption of uniform relative separation, as far as quasiconformal uniformization is concerned, is illustrated by~\cite[Proposition~3.6]{Nt20}. 
%This proposition states that, for planar Sierpi\'nski carpets of measure zero, a carpet with, e.g., square complementary components can be quasiconformally uniformized by a round carpet if and only if the complementary components are uniformly relatively separated. 
Finally, an interesting feature of Theorem~\ref{T:surf} is that a regularity requirement on $f$, which is an analytic condition, along with geometric assumptions on the minimal set yield a dynamical conclusion. This is in contrast to~\cite[Theorem~A]{KM10}, where the assumption on topological entropy is dynamical.
A generalization of Theorem~\ref{T:surf}
%, as follows from Lemma~\ref{L:UQC} and Proposition~\ref{P:Tukia}, 
is the following theorem.
\begin{theorem}\label{T:surf2}
Let $\mathcal M$ be a closed Riemann surface that is either hyperbolic  or flat  2-torus $\T^2$. Then there is no quasiconformal homeomorphism $f$ of $\mathcal M$ that permutes a dense collection of uniformly relatively separated uniform quasidisks, and such that $f$ preserves a measurable bounded conformal structure on the minimal set $\Lambda$.
\end{theorem}
This theorem gives a negative answer to the second question in~\ref{Q:NS} for more ge\-ne\-ral surfaces and under a much weaker topological and regularity assumptions on the automorphism $f$, albeit with uniform geometric assumptions on the minimal set $\Lambda$ of $f$.

The proofs of the main results, Theorem~\ref{T:tori}, Theorem~\ref{T:tori2}, Theorem~\ref{T:surf}, and Theorem~\ref{T:surf2} require some preparation as well as invocation of previously known results, notably the quasiconformal uniformizaton by M.~Bonk~\cite[Theorem~1.1]{Bo11}. 
The main tool used in the proof of Theorem~\ref{T:tori} and in establishing a non-example of  Proposition~\ref{P:Ex}  is a certain version of transboundary modulus discussed in Section~\ref{S:TMS}. The transboundary modulus or extremal length in the classical setting of multiply connected domains was introduced by O.~Schramm~\cite{Sch95}, and a certain discrete modification of it, called carpet modulus, was used successfully in the fractal setting in~\cite{Bo11} and~\cite{BM13}. The novelty in the present paper is that the transboundary modulus is used for fractals   where we allow for continuous parts of mass distributions to account for the case when minimal sets have positive measure. This is crucial for Proposition~\ref{P:Ex}.
Other results employed, even though well known for maps between open regions in $\R^n$ or the sphere $\hC$, are used here in less traditional setups of maps between fractals. An example is Proposition~\ref{P:Tukia}, used in the proof of Theorem~\ref{T:surf}. It is well known for  quasiconformal groups on open sets since the late 1970s or early 1980s from the works of  D.~Sullivan~\cite{Su78} and P.~Tukia~\cite{Tu80}, but it has not been presented in the fractal setting addressed here. Also, the proof of Theorem~\ref{T:tori2} uses essentially  the smooth Liouville Theorem  on conformal maps in higher dimensions, see, e.g.,~\cite{IM01}, but for positive measure fractals rather than for domains in $\R^n$.
While the proofs in the new settings often follow similar lines, the author felt it necessary and justified to include them or at least provide details where the two settings differ. However, if the reader is well familiar with the results in more established circumstances,  there is not much harm in skipping some of the proofs here, e.g., the proofs of Proposition~\ref{P:Extremaldistr}, as well as the auxiliary lemmas leading to its proof, and Proposition~\ref{P:Tukia}. Also, even though Lemma~\ref{L:UQC} appeared in the work of F.~Kwakkel and V.~Markovic~\cite[Lemma~2.3]{KM10}, a proof is included here since it is important and short.

\section{Quasiconformal and related maps}\label{S:Prelim}

\noindent
Throughout, we denote by $\hC$ the Riemann sphere, $\C$ the complex plane, and $\D$ the unit disk in $\C$. Also, by $\T$ we denote the unit circle.
Whenever convenient, we identify $\hC$ with $\C\cup\{\infty\}$ via the stereographic projection.
In what follows, we typically assume that a given Riemannian manifold $\mathcal M$ is either $\R^n$ with the Euclidean metric, $n$-torus $\T^n,\ n\ge2$, endowed with the metric induced by the Euclidean metric of $\R^n,\ n\ge2$, referred to as a \emph{flat} $n$-torus, or a Riemann surface endowed with the metric of constant curvature $1, 0$, or $-1$, coming from the corresponding metric of the universal cover $\hC,\C$, or $\D$, respectively. As is standard, we refer to metric of constant curvature 1 as \emph{spherical}, constant curvature 0 as \emph{flat}, and constant curvature $-1$ as \emph{hyperbolic}. 

A non-constant orientation preserving  homeomorphism  $f\: U \to V$
between  open  subsets $U$ and $V$ of a Riemannian $n$-manifold $\mathcal M$ is called \emph{quasiconformal} if the map $f$ is in the Sobolev space $W_{\rm loc}^{1,n}$
and there exists a positive constant $K$ such that the formal differential matrix $Df =(\partial f_j/\partial x_i)$
satisfies
\begin{equation}\label{E:Kqc}
\sup\left\{\|Df(p)(u)\|^n\:\ u\in\R^n,\ \|u\|\le1\right\}\leq K\cdot \det\left(Df(p)\right)
\end{equation}
for almost every $p \in U$. 
The assumption $f\in W_{\rm loc}^{1,n}$ means that $f$ and the first distributional partial derivatives of $f$ are locally in the Lebesgue space $L^n$.  One also refers to $f$ as $K$-\emph{quasiconformal}, and $K$ is called a \emph{dilatation} of $f$. If $K$ can be chosen to be 1 in Equation~\eqref{E:Kqc}, the map $f$ is called \emph{conformal}. In the case $f$ is conformal at a point $p$, we denote $Df(p)$ by $f'(p)$.
It is well known that compositions and inverses of quasiconformal maps are quasiconformal, and the same holds for conformal maps.
If $X\subseteq U$ is a measurable subset and $H$ is a collection of quasiconformal maps $h\: U\to V$ such that there exists a positive constant $K$ with each $h\in H$ satisfying Equation~\eqref{E:Kqc} at almost every point of $X$, then we say that the elements of $H$ are \emph{uniformly quasiconformal along} $X$. If $K=1$, elements of $H$ are said to be \emph{conformal along} $X$.

A homeomorphism $f\: X\to \widetilde X$ between metric spaces $(X, d_X)$ and $(\widetilde X, d_{\widetilde X})$ is called \emph{quasisymmetric} if there exists a \emph{distortion homeomorphism} $\eta\:[0,\infty)\to[0,\infty)$, such that
$$
\frac{d_{\widetilde X}(f(o),f(p))}{d_{\widetilde X}(f(o),f(q))}\leq\eta
\bigg(\frac{d_X(o,p)}{d_X(o,q)}\bigg),
$$
for every triple of distinct points $o,p$, and $q$ in $X$. We also call $f$ an $\eta$-\emph{quasisymmetric} map in this case. Geometrically, for a map 
$f$ to be quasisymmetric it means that the image under $f$ of every ball $B$ in $X$ centered at any point $p$ contains a ball in $\widetilde X$ of radius $r$ and is contained in a ball of radius $R$, both centered at $f(p)$, and such that $R/r\leq C$, where $C$ does not depend on $B$. Quasisymmetric maps are global analogues of conformal maps in the context of arbitrary metric spaces, in particular those spaces that do not possess any meaningful infinitesimal structures. Quasisymmetric homeomorphisms between open sets in $\C$ or $\hC$ are quasiconformal. As a partial converse, global quasiconformal maps of $\C$ or $\hC$ are quasisymmetric. See~\cite{He01} for background on quasisymmetric maps. 

Recall, a domain $D$ in the plane $\C$ or the sphere $\hC$ is a  \emph{quasidisk}, or a $K$-\emph{quasidisk},  if it is the image of the unit disk $\D$ under a $K$-quasiconformal map of $\C$ or $\hC$, respectively. Likewise, a \emph{quasicircle}, or a $K$-\emph{quasicircle}, is the image of the unit circle $\T$ under a $K$-quasiconformal map of $\C$ or $\hC$, respectively. The boundary curve of a quasidisk is thus a {quasicircle}. 

More generally, a \emph{quasicircle} in a metric space  $X$ is a Jordan curve $J$ in $X$ that is the image of $\T$ under a quasisymmetric map. 
In a Riemannian $n$-manifold $\mathcal M$ (or, more generally, in a doubling metric space), a Jordan curve $J$ is a quasicircle if and only if the following \emph{two point condition} is satisfied: there exists $C>0$ such that 
\begin{equation}\label{E:BT}
{\rm diam}(J_{pq})\leq C\cdot d_{\mathcal M}(p,q),\quad {\rm for\ all}\ p,q\in J,
\end{equation}
where $J_{pq}$ is the smaller diameter subarc of $J$ with endpoints $p$ and $q$, and ${\rm diam}(J)$ stands for the diameter of $J$; see~\cite{TV80}. 
We say that a family $\{J_i\}_{i\in I}$ of Jordan curves in a metric space $X$ consists of \emph{uniform quasicircles} if there exists a distortion function $\eta$ such that each curve $J_i$ is the image of $\T$ under an $\eta$-quasisymmetric map $f_i,\ i\in I$. 
In a $n$-manifold $\mathcal M$, a collection $\{J_i\}_{i\in I}$ of Jordan curves consists of uniform quasicircles if and only if~\eqref{E:BT} holds for all $J_i,\ i\in I$, with the same $C>0$. 

Let $\mathcal M$ be a Riemann surface endowed with metric of constant curvature $1, 0$, or $-1$. A \emph{quasidisk} $D$ in $\mathcal M$ is an open topological disk 
%such that the closure $\overline D$ is a closed topological disk that is embedded in $\mathcal M$, and its 
with boundary $J=\dee D$ being a quasicircle. 
A collection of domains $\{D_i\}_{i\in I}$ in $\mathcal M$ is said to consist of \emph{uniform quasidisks} if each $D_i,\ i\in I$, is a quasidisk and the boundaries $\{J_i\}_{i\in I},\ J_i=\partial D_i,\ i\in I$, are uniform quasicircles. A collection of continua (i.e., compact connected sets that consist of more than one point) $\{K_i\}_{i\in I}$ in $\mathcal M$ is said to be \emph{uniformly relatively separated} if there exists $\delta>0$ such that
$$
\Delta(K_i,K_j)=\frac{{\rm dist}(K_i,K_j)}{\min\{{\rm diam}(K_i),{\rm diam}(K_j)\}}\ge\delta,\quad {\rm for\ all}\ i, j,\ i\neq j.
$$
Here, the distance ${\rm dist}(K_i,K_j)=\inf\{d_{\mathcal M}(p,q)\: p\in K_i, q\in K_j\}$.
Quasisymmetric maps clearly take collections of uniform quasidisks to collections of uniform quasidisks, quantitatively. Also, quasisymmetric maps quantitatively preserve the uniform relative separation property, since, if $f$ is an $\eta$-quasisymmetric map, then
$$
\Delta(f(A), f(B))\le\eta\left(2\Delta(A,B)\right),
$$
where $A$ and $B$ are disjoint continua; see, e.g., \cite{He01}.

\section{Transboundary modulus}\label{S:TMS}

\noindent
The  following notion of modulus is inspired by O.~Schramm's transboundary extremal length, introduced in~\cite{Sch95}.

We fix $n\in\N,\ n>1$, and let $\mathcal M$ be a Riemannian $n$-manifold,  and $\sigma$ the Riemannian volume measure on $\mathcal M$. 
Recall, if $\Gamma$ is a path (or curve) family in $\mathcal M$, the $n$-\emph{modulus} of $\Gamma$ is defined as
$$
{\rm mod}_n(\Gamma)=\inf\left\{\int_{\mathcal M} \rho^n d\sigma\right\},
$$
where the infimum is taken over all non-negative measurable functions $\rho$ on $\mathcal M$, called \emph{mass densities}, that are \emph{admissible} for $\Gamma$, i.e.,
$$
\int_\gamma\rho\, ds\ge1,\quad {\rm for\ all}\ \gamma\in\Gamma,
$$
where $ds$ denotes the arc-length element. More generally, the notion of modulus for a curve family can be defined on arbitrary metric measure spaces.

Let $\Lambda\subset \mathcal M$ be a compact set with $\mathcal M\setminus\Lambda=\cup_{i\in I}D_i$, where $\{D_i\}_{i\in I}$ is a non-empty collection of disjoint complementary components of $\Lambda$ in $\mathcal M$.
If
$\Gamma$ is  a family of paths in
$\mathcal M$, then we  define  the \emph{transboundary modulus} of $\Gamma$ with respect to $\Lambda$, denoted by ${\rm mod}_{\Lambda}(\Gamma)$,  as follows. Let  $\rho$ be a \emph{mass distribution}, i.e., a function $\rho\colon\Lambda\cup\{D_i\}_{i\in I}\to[0,+\infty]$, where 
the restriction $\rho|_{\Lambda}$ is a measurable non-negative function, and $\rho(D_i)$ is a non-negative number for each $i\in I$. We refer to $\rho|_{\Lambda}$ as the \emph{continuous part} of a mass distribution $\rho$ and to $\{\rho(D_i)\}_{i\in I}$ as the \emph{discrete part}.
A mass distribution $\rho$ is \emph{admissible} for
${\rm mod}_\Lambda(\Gamma)$ if  
there exists a family $\Gamma_0\subseteq \Gamma$ of $n$-modulus zero, such that   
\begin{equation}\label{E:Adm}
\int_{\gamma\cap\Lambda}\rho\, ds+\sum_{\gamma\cap D_i\neq\emptyset}\rho(D_i)\ge1,
\end{equation}
for  every path
$\gamma\in\Gamma\setminus \Gamma_0$. We call $\Gamma_0$ an  {\em exceptional  family} for $\rho$. Now we set 
\begin{equation}\label{E:Modulus}
{\rm mod}_{\Lambda}(\Gamma)=\inf\left\{{\rm mass}(\rho)\right\},
\end{equation}
where the infimum is taken over all mass distributions $\rho$ that 
are admissible for ${\rm mod}_{\Lambda}(\Gamma)$, and 
\begin{equation}\label{E:Mass}
{\rm mass}(\rho)=\int_{\Lambda} \rho^n d\sigma+\sum_{i\in I}\rho(D_i)^n
\end{equation}
is the \emph{total mass} of $\rho$.
Excluding an exceptional curve family $\Ga_0$  ensures that in many interesting cases
and for some relevant  path families $\Ga$
an admissible mass distribution exists and  $0<{\rm mod}_\Lambda(\Gamma)<\infty$. We think of a mass distribution $\rho$ as an element of the direct sum Banach space $L^n\oplus l^n$, where $L^n=L^n(\Lambda)$, and $l^n$ consists of all sequences $(\rho_i)_{i\in I}$ with $\sum_{i}|\rho_i|^n<\infty$. The norm of $\rho\in L^n\oplus l^n$ is given by
$$
\|\rho\|_{L^n\oplus l^n}=\left(\int_{\Lambda}|\rho|^n d\sigma+\sum_{i\in I}|\rho(D_i)|^n\right)^{1/n}.
$$
If the set $\Lambda$ has measure zero, then we ignore the continuous part of distribution $\rho$ in~\eqref{E:Adm} and~\eqref{E:Mass}. 

Transboundary modulus has some familiar properties of an $n$-mo\-du\-lus, with almost identical proofs, e.g., monotonicity and subadditivity. Na\-me\-ly, \emph{monotonicity} means that if $\Gamma'\subseteq\Gamma$, then
$$
{\rm mod}(\Gamma')\le{\rm mod}_\Lambda(\Gamma).
$$
\emph{Subadditivity} means that if $\Gamma=\cup_{k\in\N}\Gamma_k$, then
$$
{\rm mod}_\Lambda(\Gamma)\le\sum_{k\in\N}{\rm mod}(\Gamma_k).
$$    

The following proposition is the main result of this section.
 \begin{proposition}\label{P:Extremaldistr}
 If $\Gamma$ is a curve family with ${\rm mod_\Lambda}(\Gamma)<\infty$, and if
 $\{D_i\}_{i\in I}$ have bounded geometry, then an extremal distribution $\rho$, i.e., an admissible distribution for which the infimum in the definition~\eqref{E:Modulus} is achieved, exists and is unique as an element of $L^n\oplus l^n$.
\end{proposition}
The proof of this proposition below mimics that of~\cite[Proposition~2.4]{BM13} and employs the uncentered maximal function on Riemannian manifolds. 

\subsection{Auxiliary lemmas}
Let $(\mathcal M, g)$ be a compact Riemannian $n$-manifold, $n>1$, with $\sigma$ denoting the Riemannian volume measure. 
Then $\sigma(\mathcal M)<\infty$.
 If $f\: \mathcal M\to\R$ is a measurable function in $L^1(\mathcal M)$, we define the \emph{uncentered maximal function} $M_f$ on $\mathcal M$ by
 $$
 M_f(x)=\sup_{x\in B}\left\{\frac{1}{\sigma(B)}\int_B|f|d\sigma\right\},
 $$ 
 where the supremum is taken over all balls $B$ in $\mathcal M$ that contain $x$, and not only the ones centered at $x$. Note that since $\mathcal M$ is assumed to be compact, we may restrict the radii of the balls $B$ to be at most ${\rm diam}(\mathcal M)$.
 If the balls in the definition of $M_f$ are centered at $x$, the corresponding maximal function is referred to as \emph{centered}. 

To prove Proposition~\ref{P:Extremaldistr} we need Lem\-ma~\ref{L:Modcomp} below (cf.~\cite[Lemma~2.3]{BM13}),  which, in turn, requires the following  lemmas.
 
%The following lemma is standard.
%shows that the uncentered maximal function is a bounded operator in $L^p(\mathcal M),\ 1<p<\infty$.  This is classical in $\R^n$,  and for the centered maximal function the  lemma in the more general setting of doubling metric measure spaces is part of~\cite[Theorem~2.2]{He01}.
%The proof in the Riemannian manifolds setting is almost identical to that for $\R^n$, however it is  included for the convenience of the reader and to make the paper more self-contained. It follows the lines of~\cite{Au12}, and the reader familiar with the notion and basic properties of maximal functions may safely skip the proof. 
 \begin{lemma}\cite[Theorem~2.2]{He01}\label{L:Lpbound}
 If $\mathcal M$ is a compact Riemannian manifold, then for each $p,\ 1<p<\infty$, there exists $C>0$ such that
 \begin{equation}\label{E:Lpbound}
 \left\|M_f\right\|_p\leq C\,\left\|f\right\|_p,\quad {\rm for\ all}\ f\in L^p(\mathcal M).
 \end{equation}
 \end{lemma}

\begin{lemma}\cite{Bo88}, \cite[Lemma~2.2]{BM13}\label{L:Boj}
Let $\{B_i\}_{i\in I}$ be a finite or countable collection of disjoint geometric balls in a compact Rieman\-nian $n$-manifold $\mathcal M$. Further,  let $\{a_i\}_{i\in I}$ be non-negative numbers. Then, for each $\lambda\ge1$ there exists $C>0$ such that
\begin{equation}\label{E:Boj}
\left\|\sum_{i\in I}a_i\chi_{\lambda B_i}\right\|_{L^n}\leq C \left\|\sum_{i\in I}a_i\chi_{B_i}\right\|_{L^n}.
\end{equation}
\end{lemma}
\begin{proof}
To prove this inequality, we let $\varphi\in L^p(\mathcal M)$ be arbitrary, where $1/p+1/n=1$. Then, if $M_\varphi$ is the uncentered maximal function of $\varphi$, we have $c>0$ depending on $\mathcal M$ with 
$$
\begin{aligned}
&\left|\int_{\mathcal M}\left(\sum_{i\in I}a_i\chi_{\lambda B_i}\right)\varphi\, d\sigma\right|=\left|\sum_{i\in I}a_i\int_{\lambda B_i}\varphi\, d\sigma\right|
\leq c\lambda^n\sum_{i\in I}a_i\int_{B_i}M_\varphi d\sigma\\
&=c\lambda^n\int_{\mathcal M}\left(\sum_{i\in I} a_i\chi_{B_i}\right)M_{\varphi} d\sigma\leq c\lambda^n\left\|\sum_{i\in I}a_i\chi_{B_i}\right\|_{L^n}\left\|M_\varphi\right\|_{L^p}.
\end{aligned}
$$
Lemma~\ref{L:Lpbound} along with the duality of $L^n$ and $L^p$ now gives the desired conclusion.
\end{proof}

\begin{lemma}\label{L:Modcomp}
If domains $D_i,\ i\in I$, in a Riemannian $n$-manifold $\mathcal M$ have bounded geometry and $\Gamma$  is an arbitrary curve family in  $\mathcal M$, then ${\rm mod}_{\Lambda}(\Gamma)=0$ implies ${\rm mod}_{n}(\Gamma)=0$.
\end{lemma}
\begin{proof}
Since $\Gamma$ cannot contain constant curves,
we have
$$
\Gamma=\cup_{k\in\N}\Gamma_k,\quad{\rm where}\quad \Gamma_k=\{\gamma\in\Gamma:\ {\rm diam}(\gamma)\ge1/k\}.
$$
From subadditivity of ${\rm mod}_n$ it follows that it is enough to show that ${\rm mod}_n(\Gamma_k)=0$ for each $k\in\N$. We fix $k\in\N$. Monotonicity of ${\rm mod}_\Lambda$ gives that  ${\rm mod}_\Lambda(\Gamma_k)=0$. Thus, for each $l\in\N$,  there exists a mass distribution $\rho_{k,l}$ with ${\rm mass}(\rho_{k,l})<1/2^l$, and an exceptional curve family $\Gamma_{k,l}$ such that
$$
\int_{\gamma\cap\Lambda}\rho_{k,l}\, ds+\sum_{i\:\gamma\cap D_i\neq\emptyset}\rho_{k,l}(D_i)\ge1,
$$ 
for  all $\gamma\in\Gamma_k\setminus\Gamma_{k,l}$.

Let $\rho_k=\sum_{l\in\N}\rho_{k,l}$ and $\Gamma_{k,\infty}=\cup_{l\in\N}\Gamma_{k,l}$. Then, ${\rm mass}(\rho_k)<\infty$ and subadditivity of ${\rm mod}_n$ gives ${\rm mod}_n(\Gamma_{k,\infty})=0$. Moreover,
\begin{equation}\label{E:inflength}
\int_{\gamma\cap\Lambda}\rho_{k}\, ds+\sum_{i\:\gamma\cap D_i\neq\emptyset}\rho_{k}(D_i)=\infty
\end{equation}
for all $\gamma\in\Gamma_k\setminus\Gamma_{k,\infty}$.

Recall, the bounded geometry assumption gives the existence of balls $B(p_i, r_i), B(p_i, R_i)$ such that 
$$
B(p_i, r_i)\subseteq D_i\subseteq B(p_i, R_i),\quad {\rm with}\quad {R_i}/{r_i}\leq C.
$$
Now we define a mass density function $\rho$ on $\mathcal M$  by the formula
$$
\rho=\rho_k+\sum_{i\in I}\frac{\rho_k(D_i)}{R_i}\chi_{B(p_i,2R_i)},
$$
where the first summand is extended  to the complement of $\Lambda$ as 0.
We first claim that for each $\gamma\in\Gamma_k\setminus\Gamma_{k,\infty}$ one has
\begin{equation}\label{E:inflength2}
\int_\gamma\rho\, ds=\infty.
\end{equation}
Indeed, for all but finitely many $i$, since ${\rm diam}(\gamma)\ge 1/k$, if $\gamma$ intersects $D_i\subseteq B(p_i, R_i)$, it intersects both complementary components of $B(p_i,2R_i)\setminus B(p_i, R_i)$. This implies
$$
\int_\gamma\chi_{B(p_i,2R_i)}ds\ge R_i,
$$
and the claim follows from Equation~\eqref{E:inflength}. Next, from Lemma~\ref{L:Boj} we conclude
\begin{equation}\label{E:finitemass}
\begin{aligned}
\int_{\mathcal M}\rho^nd\sigma&\leq 2^{n-1}\left(\int_\Lambda \rho_k^n d\sigma+C^n\sum_{i\in I}\frac{\rho_k(D_i)^n}{R_i^n}\sigma\left(B(p_i, r_i)\right)\right)\\&\leq C(\mathcal M)\cdot{\rm mass}(\rho_k)<\infty,
\end{aligned}
\end{equation}
where $C(\mathcal M)$ is a constant that depends on the manifold $\mathcal M$. Then, \eqref{E:inflength2} and~\eqref{E:finitemass} imply ${\rm mod}_n(\Gamma_k\setminus\Gamma_{k,\infty})=0$. Hence,
$$
{\rm mod}_n(\Gamma_k)\leq{\rm mod}_n(\Gamma_{k,\infty})+{\rm mod}_n(\Gamma_k\setminus\Gamma_{k,\infty})=0.
$$
\end{proof}

\subsection{Proof of Proposition~\ref{P:Extremaldistr}} Let $(\rho_k)_{k\in\N}$ be a minimizing sequence of admissible mass distributions for ${\rm mod}_\Lambda(\Gamma)$, i.e.,
$$
{\rm mass}(\rho_k)\to{\rm mod}_\Lambda(\Gamma),\quad k\to\infty.
$$
The condition ${\rm mod}_\Lambda(\Gamma)<\infty$ allows us to assume that ${\rm mass}(\rho_k)\le C$ for some constant $C>0$ and all $k\in\N$. 
By the Banach--Alaoglu Theorem, there exists a subsequence of $(\rho_k)_{k\in\N}$ that weakly converges to a mass distribution $\rho\in L^n\oplus l^n$.
In particular, after possibly passing to a subsequence, we may assume that the limits
$$
\rho(D_i)=\lim_{k\to\infty}\rho_k(D_i)
$$
exist for all $i\in I$. We claim that the mass distribution $\rho$ is extremal. 

One inequality, namely ${\rm mass}(\rho)\leq{\rm mod}_\Lambda(\Gamma)$ follows from the weak lower semicontinuity of norms.
To show the reverse inequality one needs to demonstrate that $\rho$ is admissible for $\Gamma$.
Since the sequence $(\rho_k)_{k\in\N}$ converges weakly to $\rho\in L^n\oplus l^n$, Mazur's Lemma~\cite[Th.\ 2, p.\ 120]{Yo80} gives that there is a sequence $(\tilde\rho_K)_{K\in\N}$ of convex combinations,
$$
\tilde\rho_K=\sum_{k=1}^K\lambda_{K,k}\rho_k,
$$
that strongly converges to $\rho$ in $L^n\oplus l^n$. 
Therefore, $(\tilde\rho_K)_{K\in\N}$ is also a minimizing sequence for  ${\rm mod}_\Lambda(\Gamma)$. Moreover, each $\tilde\rho_K,\ K\in\N$, is trivially admissible for $\Gamma$. 
An exceptional curve family $\widetilde\Gamma_K$ for $\tilde\rho_K$ is the union of exceptional curve families for $\rho_k,\ k=1,2,\dots, K$.
 By possibly passing to a subsequence, we may assume that
 $$
 \left\|\tilde\rho_K-\rho\right\|_{L^n\oplus l^n}\le\frac{1}{2^K},\quad K\in\N. 
 $$
 Let 
 $$
 \Gamma_K'=\left\{\gamma\in\Gamma\: \int_{\gamma\cap \Lambda}|\tilde\rho_K-\rho|\, ds+\sum_{i\: \gamma\cap D_i\neq\emptyset}|\tilde\rho_K(D_i)-\rho(D_i)|\ge\frac{1}{K}\right\},
 $$
 and 
 $$
 \Gamma'=\cap_{k\in\N}\cup_{K\ge k}\Gamma'_K.
 $$
From this definition we have that if $\gamma\in\Gamma\setminus\Gamma'$, then
 $$
 \limsup_{K\to\infty}\left(\int_{\gamma\cap \Lambda}|\tilde\rho_K-\rho|\, ds+\sum_{i\: \gamma\cap D_i\neq\emptyset}|\tilde\rho_K(D_i)-\rho(D_i)|\right)=0.
 $$
 The mass distributions
 $$
 \rho_k'=\sum_{K=k}^\infty K|\tilde\rho_K-\rho|,\quad k\in\N,
 $$
 are admissible for $\Gamma'$, and since $\left\|\tilde\rho_K-\rho\right\|_{L^n\oplus l^n}\le{1}/{2^K}$, we have that ${\rm mass}(\rho_k')\to0,\ k\to\infty$. Thus ${\rm mod}_\Lambda(\Gamma')=0$. Lemma~\ref{L:Modcomp} then gives ${\rm mod}_n(\Gamma')=0$. We conclude that $\Gamma_\infty=\Gamma'\cup\left(\cup_{K\in\N}\widetilde\Gamma_K\right)$ is an exceptional family for $\rho$. Indeed, by subadditivity of ${\rm mod}_n$ it has modulus 0, and for each $\gamma\in\Gamma\setminus\Gamma_\infty$, since each $\tilde\rho_K$ is admissible,
 $$
 \begin{aligned}
 &\int_{\gamma\cap\Lambda}\rho\, ds+\sum_{i\:\gamma\cap D_i\neq\emptyset}\rho(D_i)\\
 &\ge1-\limsup_{K\to\infty}\left(\int_{\gamma\cap \Lambda}|\tilde\rho_K-\rho|\, ds+\sum_{i\: \gamma\cap C_i\neq\emptyset}|\tilde\rho_K(C_i)-\rho(C_i)|\right)=1.
 \end{aligned}
 $$
This completes the proof of existence of an extremal mass distribution for ${\rm mod}_\Lambda(\Gamma)$. 

 A simple convexity argument implies that an extremal distribution $\rho$ is also unique. Indeed, if there exists another, different, extremal mass distribution $\rho'$, then
 the average $(\rho+\rho')/2$ is also admissible for $\Gamma$. The strict convexity of $L^n\oplus l^n,\ 1<n<\infty$, now gives that the average has a strictly smaller total mass, a contradiction. 
\qed

To finish this section, we prove the following conformal invariance  lemma that makes transboundary modulus particularly useful. If $f$ is a homeomorphism of  $\mathcal M$ that is conformal along its minimal set $\Lambda=\mathcal M\setminus\cup_{i\in I}D_i$, and $\rho$ is a mass distribution on $\Lambda\cup\{D_i\}_{i\in I}$, we define the \emph{pullback distribution} $f^*(\rho)$ to be $\rho(f)\|f'\|$ on $\Lambda$ and $f^*(\rho)(D_i)=\rho(f(D_i)),\ i\in I$. 
\begin{lemma}\label{L:Confinv}
Let $f$ be quasiconformal homeomorphism of $\mathcal M$ that is conformal along its minimal set $\Lambda$.  Let $\Gamma$ be an arbitrary curve family in $\mathcal M$. If $\rho$ is the extremal mass distribution for $f(\Gamma)$, where $f(\Gamma)=\{f(\gamma)\: \gamma\in\Gamma\}$, then the pullback distribution $f^*(\rho)$ is the extremal mass distribution for $\Gamma$. 
Moreover, 
$$
{\rm mod}_\Lambda(f(\Gamma))={\rm mod}_\Lambda(\Gamma).
$$ 
%In particular, if $\Lambda$ has measure zero, the above holds for any quasiconformal homeomorphism $f$ of $\mathcal M$ for which $\Lambda$ is the minimal set.
\end{lemma}
\begin{proof}
To make the desired conclusion, we use the assumption that $f$ is quasiconformal. Indeed, let $\Gamma_0\subseteq f(\Gamma)$ be an exceptional curve family for $\rho$, and let $\Gamma_0'=f^{-1}(\Gamma_0)$ be the preimage family. Then we have ${\rm mod}_n(\Gamma_0')=0$ because  quasiconformal maps preserve vanishing of $n$-modulus. The latter is a simple consequence of the geometric definition of a quasiconformal map in $\R^n$; see~\cite{He01}. In addition, let $\Gamma_0''$ be the family of paths on which $f$ is not absolutely continuous. Again, since $f$ is quasiconformal, we have ${\rm mod}_n(\Gamma_0'')=0$, and thus ${\rm mod}_n(\Gamma_0'\cup \Gamma_0'')=0$. If $\gamma\in\Gamma\setminus(\Gamma_0'\cup\Gamma_0'')$, then, in particular, $f(\gamma)\in\Gamma\setminus\Gamma_0$, and hence
$$
\begin{aligned}
&\int_{\gamma\cap\Lambda}\rho(f)\|f'\|\, ds+ \sum_{\gamma\cap D_i\neq\emptyset}\rho(f(D_i))\\
&=\int_{f(\gamma)\cap\Lambda}\rho\, ds+\sum_{f(\gamma)\cap f(D_i)\neq\emptyset}\rho(f(D_i))\\
&=\int_{f(\gamma)\cap\Lambda}\rho\, ds+\sum_{f(\gamma)\cap D_i\neq\emptyset}\rho(D_i)\ge1,
\end{aligned}
$$ 
i.e., the pullback mass distribution $f^*(\rho)$ is admissible for $\Gamma$. The total mass of $f^*(\rho)$ is
$$
\int_\Lambda\rho(f)^n\|f'\|^nd\sigma+\sum_{i\in I}\rho(f(D_i))^n=\int_\Lambda\rho^nd\sigma+\sum_{i\in I}\rho(D_i)^n,
$$
i.e., it is equal to ${\rm mass}(\rho)$. Therefore, 
$$
{\rm mod}_\Lambda(\Gamma)\le{\rm mod}_\Lambda(f(\Gamma)).
$$
The converse inequality and indeed the full lemma follow from the facts that the inverse of a quasiconformal map is quasiconformal, and the inverse of a conformal map is conformal.  
\end{proof}

\section{Proof of Theorem~\ref{T:tori}}\label{S:Tori}

\noindent
Assume for contradiction that there exists a Denjoy type quasiconformal homeomorphism $f$ whose minimal set $\Lambda$ has measure zero and wandering domains have bounded geometry. 
Consider  $\Gamma_h$, a path family on $n$-torus that consists of all \emph{horizontal} paths.
I.e., if we write $\T^n=\R^n/\Z^n=(\R/\Z)\times(\R^{n-1}/\Z^{n-1})$, then these are simple closed curves isotopic to the circle 
$$
C_1=\{(x_1,0,\dots,0)\: x_1\in\R/\Z\}\subset\T^n.
$$ 
There is a natural \emph{orthogonal projection} $\pi_1$ of $\T^n$ onto $C_1$, namely, $\pi_1(x_1,x_2,\dots,x_n)=(x_1, 0,\dots, 0),\ x_i\in\R/\Z,\ i=1,2,\dots, n$.

The assumption of bounded geometry  implies the following lemma.
\begin{lemma}\label{L:mogafi}
If $f$ and $\Lambda$ are as above, for the horizontal path family $\Gamma_h$ one has ${\rm mod}_\Lambda(\Gamma_h)<\infty$.
\end{lemma}
\begin{proof}
Define a mass distribution $\rho$ to be 0 on $\Lambda$ and $\rho(D_i)={\rm diam}(D_i)$,  the diameter value of each complementary component $D_i,\ i\in I$, of $\Lambda$. Such $\rho$ is 
admissible for $\Gamma_h$. Indeed, let $\Gamma_0$ be the subfamily of $\Gamma_h$ whose curves spend positive or infinite length in $\Lambda$. Since $\Lambda$ is assumed to have measure zero, we have ${\rm mod}_n(\Gamma_0)=0$. Now, if $\gamma\in\Gamma_h\setminus\Gamma_0$, let $\{D_i\:\ i\in I_\gamma\}$ be the collection of all complementary components of $\Lambda$ in $\T^n$ that intersect $\gamma$. Since $\gamma$ spends zero length in $\Lambda$, the orthogonal projections $\pi_1(D_i),\ i\in I_\gamma$,   cover $C_1$ with a possible exception of a set of length 0. Therefore, $\sum_{i\in I_\gamma}{\rm diam}(D_i)\ge\sum_{i\in I_\gamma}{\rm diam}(\pi_1(D_i))\ge1$, i.e., $\rho$ is admissible. The assumption of bounded geometry for $D_i,\ i\in I$, implies that the total mass of $\rho$ is bounded above by $C(n)$, where $C(n)$ is a constant that depends on the dimension $n$ and the constant $C$  from~\eqref{E:Bddgeom}.
\end{proof}

Proposition~\ref{P:Extremaldistr} now implies that the unique extremal mass distribution $\rho_{h}$ for $\Gamma_h$ exists. Since $f$ is isotopic to the identity, the path family $f(\Gamma_h)=\{f(\gamma)\: \gamma\in\Gamma_h\}$ coincides with the horizontal path family $\Gamma_h$. In particular, 
$$
{\rm mod}_\Lambda(f(\Gamma_h))={\rm mod}_\Lambda(\Gamma_h).
$$
The last equality holds for any homeomorphism of $\T^n$ isotopic to the identity. Since $f$ is quasiconformal, Lemma~\ref{L:Confinv}  along with the uniqueness of the extremal mass distribution give more, namely that 
$$
\rho_h(f(D_i))=\rho_h(D_i),\ i\in I.
$$

 But  $D_i,\ i\in I$, are wandering domains, and thus we must have that $\rho(D_i)$ is a constant $C$ for each of infinitely many domains $D_i$ in the same orbit under the map $f$. However ${\rm mass}(\rho_{h})<\infty$, and thus $C=0$ for each such orbit, and therefore $\rho$ is identically zero. This is impossible because a mass distribution that is identically zero cannot be admissible. The contradiction finishes the proof.
\qed

\section{The Liouville Theorem and proof of Theorem~\ref{T:tori2}}\label{S:LT}

\noindent
The main ingredient in the proof of Theorem~\ref{T:tori2} is a version of~\cite[Theorem~2.3.1]{IM01} in the context of fractals.
Below, for a matrix $A$ we denote its transpose by $A^t$. Also, in what follows, by a \emph{M\"obius map} or \emph{transformation} in $\hC$ or $\R^n\cup\{\infty\}$ we mean a finite composition of reflections in spheres. Note that a M\"obius transformation may be orientation preserving or reversing. 

\begin{theorem}\label{T:Liouv}
Let $\Lambda$ be a closed subset of $\R^n$ that is connected and locally rectifiably connected. In addition, assume that the Lebesgue density points of $\Lambda$ 
%that are not on the boundary of $\R^n\setminus\Lambda$ 
are dense in $\Lambda$.   
Let $f\:\R^n\to\R^n$ be a $\mathcal C^3$-diffeomorphism such that $f$ preserves the standard conformal structure on $\Lambda$. 
%and the Jacobian determinant $J(x,f)=\det Df(x)$ does not change sign in $\Omega$, 
Then $f$ is a M\"obius transformation on $\Lambda$. 
\end{theorem}  
\begin{proof}
The proof is
almost verbatim the proof as presented in~\cite{IM01}. Because of this, we do not provide all the details below and only highlight the main steps and differences. 

For $f$ to preserve the standard conformal structure on $\Lambda$ it means it
satisfies the Cauchy--Riemann system
$$
(Df(x))^t\cdot Df(x)=|\det Df(x)|^{2/n}\id,\quad {\rm at\ each}\  x\in\Lambda.
$$ 
To simplify notations, we consider a more general system 
\begin{equation}\label{E:Conf1}
(Df(x))^t\cdot Df(x)=G(x),
\end{equation}
where $G(x)=[g_{ij}(x)]$ is a positive definite symmetric $n\times n$ matrix with entries that are twice continuously differentiable. Let $G^{-1}(x)=[g^{ij}]$ denote the inverse of $G(x)$. We can rewrite Equation~\eqref{E:Conf1} as
\begin{equation}\label{E:Conf2}
Df(x)\cdot G^{-1}(x)\cdot (Df(x))^t=\id.
\end{equation}
Denoting the coordinate functions of $f$ by $f^1, f^2,\dots, f^n$ and the partials of $f^i$ by $f^i_j=\dee f^i/\dee x_j,\ i, j=1,2,\dots, n$, we can write Equations~\eqref{E:Conf1}, \eqref{E:Conf2} as
\begin{equation}\label{E:Conf3}
f^\alpha_i f^\alpha_j=g_{ij},\quad g^{\nu\mu}f^i_\nu f^j_\mu=\delta_{ij},
\end{equation}
respectively,
where we use Einstein's summation convention, i.e., we sum over repeated indices, and $\delta_{ij}$ denotes the Kronecker symbol. After differentiating the first system of equations in~\eqref{E:Conf3} with respect to $x_k$ and  manipulating algebraically, we obtain
\begin{equation}\label{E:Conf4}
f_{jk}=\Gamma^\nu_{jk}f_\nu,
\end{equation}
where $f^i_{jk}=\dee^2 f^i/(\dee x_j\dee x_k)$, and
$$
\Gamma^\nu_{jk}=\frac12 g^{i\nu}\left(\frac{\dee g_{ij}}{\dee x_k}+\frac{\dee g_{ik}}{\dee x_j}-\frac{\dee g_{jk}}{\dee x_i}\right)
$$ 
are the Christoffel symbols of $G$. 

A consequence of~\eqref{E:Conf4} is that if $G$ is constant on $\Lambda$, then $f$ is a linear transformation when restricted to  $\Lambda$. This observation would be trivial if $\Lambda$ were a domain. However, for $\Lambda$ as in the statement of Theorem~\ref{T:Liouv}  one needs to exercise some care. Indeed, first of all, the components $g_{ij},\ i,j=1,2,\dots, n$, of $G$ being constants on a closed set does not in general imply that all the partial derivatives  $\dee g_{ij}/\dee x_k$ are zero on that set. This is implied however at each Lebesgue density point of $\Lambda$. 
%that is not on the boundary of $\R^n\setminus\Lambda$. 
The assumption that such points are dense in $\Lambda$ along with the $\mathcal C^3$-regularity assumption give that all $\dee g_{ij}/\dee x_k$ are zero at each point of $\Lambda$. 
Hence all $f^i_{jk}$ are identically zero in $\Lambda$. 
The claim that the restriction of $f$ to $\Lambda$ is linear now follows from the Fundamental Theorem for Line Integrals because $\Lambda$ is assumed to be closed, connected, and locally rectifiably connected.

In particular, if $f$ satisfies the Cauchy--Riemann system in $\Lambda$ with constant Jacobian determinant, then, when restricted to $\Lambda$, the map $f$ has the form
\begin{equation}\label{E:Conf4.5}
f(x)=b+\alpha A x,
\end{equation}
where $b\in\R^n$, $\alpha$ is a positive constant, and $A$ is an orthogonal matrix.

Differentiating~\eqref{E:Conf4} with respect to $x_l$ and using~\eqref{E:Conf3}, one obtains
\begin{equation}\label{E:Conf5}
f_{jkl}=\left(\frac{\dee\Gamma^\nu_{jk}}{\dee x_l}+\Gamma^\mu_{jk}\Gamma^\nu_{\mu l}\right) f_\nu,
\end{equation}
where $f_{jkl}$ denotes $\dee^3 f/(\dee x_j\dee x_k\dee x_l)$. Permuting $k$ and $l$ and using the assumption that $f$ is in $\mathcal C^3$, we obtain
\begin{equation}\label{E:Conf6}
R^\nu_{jlk} f_\nu=0,
\end{equation}
where
$$
R^\nu_{jlk}=\frac{\dee\Gamma^\nu_{jk}}{\dee x_l}-\frac{\dee\Gamma^\nu_{jl}}{\dee x_k}+\Gamma^\mu_{jk}\Gamma^\nu_{\mu l}-\Gamma^\mu_{jl}\Gamma^\nu_{\mu k}
$$
is the Riemannian curvature tensor. The map $f$ being a diffeomorphism implies that the vectors $f_\nu,\ \nu=1,2\dots, n$, are linearly independent, and thus~\eqref{E:Conf6} gives
$$
R^\nu_{jlk}=0.
$$

The next claim is that for $f$ satisfying the Cauchy--Riemann system on $\Lambda$, either $J(x,f)$ is constant or 
$$
J(x,f)=r^{2n}\|x-a\|^{-2n},\quad x\in\Lambda,
$$
for some constant $r>0$ and a constant vector $a\in\R^n$. Indeed, suppose that $J(x,f)$ is non-constant. Since
$$
G(x)=\left(\det Df(x)\right)^{2/n}\id,\quad x\in\Lambda,
$$ 
we have that the Ricci curvature tensor $R_{ij}=R^\nu_{i\nu j}$ vanishes on $\Lambda$. For the function
$$
P(x)=J(x,f)^{-1/n},
$$
vanishing of the Ricci curvature tensor can be written on $\Lambda$ as
\begin{equation}\label{E:Conf7}
2P\, P_{ij}=\|\nabla P\|^2\delta_{ij},
\end{equation}
where, as above, the lower indices in $P_{ij}$ denote the corresponding partial derivatives. Differentiating~\eqref{E:Conf7} with respect to $x_k$ and ma\-ni\-pu\-la\-ting algebraically again one arrives at
$$
P_{ijk}(x)=0,\quad x\in\Lambda,
$$
for all indices $i,j,k=1,2,\dots, n$. This combined with the assumption that $\Lambda$ is closed, connected, and locally rectifiably connected  shows that the second partial derivatives $P_{ij}$ are constant in $\Lambda$, and, in view of~\eqref{E:Conf7}, one must have
$$
P_{ij}(x)=c_{ij}=c\, \delta_{ij},\quad x\in\Lambda,\ i,j=1,2,\dots, n,
$$
for some constant $c\ge0$. If $c=0$, Equation~\eqref{E:Conf7} gives that $P$, and hence the Jacobian determinant, are constant in $\Lambda$. Thus we may assume that $c>0$ and denote $c=2 r^{-2}$, where $r>0$. Integrating along rectifiable paths in $\Lambda$ one then has
$$
\nabla P=2r^{-2}(x-a),\quad x\in\Lambda,
$$
for some $a\in\R^n$, and hence
$$
J(x,f)=P(x)^{-n}=r^{2n}\|x-a\|^{-2n},\quad x\in\Lambda.
$$

If $J(x,f)$ is a positive constant, we know from the above that $f$ is a linear map of the form~\eqref{E:Conf4.5}, where $A$ is an orthogonal matrix. In the case
$$
J(x,f)=r^{2n}\|x-a\|^{-2n},
$$
if one writes $f=F\circ g$, where $g$ is the reflection in the sphere centered at $a$ with radius $r$, then a simple computation shows that $F$ is a solution to  the Cauchy--Riemann system in $g(\Lambda)$ and has a constant Jacobian there. Hence $F$ has the form~\eqref{E:Conf4.5}, and therefore 
$$
f(x)=b+\frac{\alpha A(x-a)}{\|x-a\|^2},\quad x\in\Lambda,
$$
a M\"obius transformation.
\end{proof}

\medskip
\noindent
We are now ready to finish the proof of Theorem~\ref{T:tori2}. Assume for contradiction that $f\in{\rm Diff}^2(\T^n)$ as in Theorem~\ref{T:tori2} exists. Let $\widetilde\Lambda$ denote the preimage of $\Lambda$ under the covering map from $\R^n$ to $\T^n$, and let $\tilde f$ denote a lift of the map $f$ under the covering map. It is immediate that the conditions of Theorem~\ref{T:Liouv} are satisfied for $\widetilde\Lambda$ and $\tilde f$. Therefore, the restriction of $\tilde f$ to $\widetilde\Lambda$  must be a M\"obius transformation. Since $\tilde f$ preserves the infinity, the map $\tilde f$ must be linear, i.e., 
$$
\tilde f(x)=b+\alpha A x,\quad x\in\Lambda.
$$

Compactness of $\T^n$ and connectedness assumptions of $\Lambda$ imply that the lifts of wandering domains of $f$ to $\R^n$ have uniformly bounded diameters.  One therefore must have $\alpha=1$, and hence $\tilde f$ is an isometry. In particular, $f$ cannot change the diameters of the wandering domains. This is impossible, i.e., no map $f$ as in Theorem~\ref{T:tori2} exists.
\qed

\medskip
\noindent
We finish this section with two comments. First, it should be possible to relax the $\mathcal C^3$-regularity assumption in Theorem~\ref{T:tori2} to Sobolev regularity $W_{\rm loc}^{1,n}$. In that case one would require the equation $(Df(x))^t\cdot Df(x)=|\det Df(x)|^{2/n}\id$ to be satisfied only almost everywhere in $\Lambda$. The proof of the main tool, namely the Liouville Theorem, is very involved though even in the case of a domain in $\R^n$. Second, the current proof does not yield itself to replacing the standard conformal structure by an arbitrary measurable bounded conformal structure because in dimension 3 and higher the Measurable Riemann Mapping Theorem is not available. In particular, it is not sufficient to only require that the wandering domains have bounded geometry. Indeed, under $\mathcal C^1$-regularity one can show that the quasiconformal dilatations of all iterates of $f$ are uniformly bounded; cf.~Lemma~\ref{L:UQC} below.   One can then argue~\cite[Corollary~21.5.1]{IM01} that in this case there exists an invariant  conformal structure under $f$; cf.~Proposition~\ref{P:Tukia} below. But, again, the lack of the Riemann Mapping Theorem does not allow to proceed.

\section{Schottky sets and groups}\label{S:SSG}

\noindent
Here we recall some notions and facts on {Schottky} sets and Schottky groups; see~\cite{BKM09} for background and more details. A \emph{Schottky set} $S$ is a closed subset of 2-sphere that can be written in the form 
\begin{equation}\label{stdS}
S=\hC\setminus \bigcup_{i\in I}B_i,
\end{equation}
 where the sets   $B_i,\ i\in I$, are pairwise disjoint 
open disks in $\hC$. It is assumed that $I$ consists of at least three elements, and, in fact, is infinite in what follows.  The boundary circles $\dee B_i,\ i\in I$, are referred to as \emph{peripheral circles}. They are characterized topologically by the fact that their removal does not separate $S$. In particular, any homeomorphism between Schottky sets must take peripheral circles to peripheral circles. It is immediate from the definition that M\"obius transformations take Schottky sets to Schottky sets.

The following is a quasiconformal uniformization result from~\cite{Bo11} that will be used in our proof of Theorem~\ref{T:surf}. 
\begin{theorem}\cite[Theorem~1.1]{Bo11}\label{T:Unif}
Suppose that $\Lambda$ is a closed subset of $\hC$ whose complement is a union of uniformly relatively separated uniform quasidisks. Then there exists a quasiconformal homeomorphism  $\beta$ of $\hC$ such that $S=\beta(\Lambda)$ is a Schottky set. 
\end{theorem}
This is a quasiconformal analogue of P.~Koebe's (see, e.g.,~\cite{Co95}), and Z.-X.~He--O.~Schramm's~\cite{HS93} conformal uniformizations of finitely connected, respectively countably connected, domains.  The proof of Theorem~\ref{T:Unif} uses Koebe's theorem and a delicate analysis that allows to control the quasiconformal dilatations of countably many quasiconformal maps using given geometry of $\Lambda$.

If  $S\sub\hC$ is  a Schottky set written as in \eqref{stdS}, then for each $i\in I$, let $R_i\: \hC\to \hC$ be the reflection in the 
peripheral circle $\partial B_i$. 
The subgroup of the group of all M\"obius transformations on $\hC$
generated by the reflections $R_i$, $i\in I$,  is the 
\emph{Schottky group} associated to $S$ and denoted by $G_S$.  The group $G_S$ is a discrete group of M\"obius transformations with a presentation 
given by the generators $R_i$, $i\in I$, and the relations $R_i^2=\id$, $i\in I$.  

We set 
\begin{equation} \label{Sinfpart}
S_{\infty}=\bigcup_{g\in G_S}g(S).
\end{equation} 
This is the \emph{dissipative part} of $G_S$ in the terminology of~\cite{Su78}, and the sets $g(S),\ g\in G_S$, form a measurable partition of $S_\infty$.
 The \emph{conservative part} of $G_S$ consists of all points that are nested in infinitely many disks that are components of the complements of $g(S),\ g\in G_S$.  
The sphere $\hC$ is then decomposed into the {dissipative} and {conservative} parts.

If $\mathcal M=\C$ or $\hC$, a \emph{Beltrami coefficient} is a complex measurable function $\mu$  with $\|\mu\|_\infty<1$. If $h\colon U\to V$ is a quasiconformal map, 
the pullback of the trivial Beltrami coefficient $\mu=0$, i.e., the standard complex structure, is the \emph{Beltrami coefficient} $\mu_h$ \emph{associated to} $h$ and given by
$$
\mu_h=
\begin{cases}
{h_{\bar z}}/{h_z},\quad {\rm at\ points\ of\ differentiability\ of}\ h\ {\rm in}\ U,\\
0,\quad {\rm elsewhere}.
\end{cases}
$$ 
The quantity  
$$
K=\frac{1+\|\mu_h\|_\infty}{1-\|\mu_h\|_\infty}.
$$
is a \emph{dilatation} of $h$.
Conversely, the Measurable Riemann Mapping Theorem states that given any Beltrami coefficient $\mu$ on $\hC$, there exists a quasiconformal map $h$ of $\hC$ such that $\mu_h=\mu$ almost everywhere. Moreover, the map $h$ is uniquely determined by a post-composition by an orientation preserving  M\"obius transformation. 

If $S$ is a Schottky set, the group $G_S$ is the associated Schottky group,  and $\mu$ is a Beltrami coefficient on $S$, i.e., a Beltrami coefficient that vanishes on $\hC\setminus S$, we can define almost everywhere a new Beltrami coefficient
$$
\mu_{S}(z)=
\begin{cases}
g^*(\mu)(z),\quad z\in g(S),\ g\in G_S,\\
0,\quad {\rm otherwise}.
\end{cases}
$$
If $g$ is a M\"obius map, the pullback $g^*(\mu)$ of a Beltrami coefficient $\mu$ by $g$ is given by
$$
g^*(\mu)=\mu(g)\cdot \overline{g_z}/g_z,\quad {\rm or}\quad g^*(\mu)=\overline{\mu(g)}\cdot g_{\bar z}/\overline{g_{\bar z}},
$$
depending on whether $g$ is orientation preserving or reversing, respectively. 
It is immediate from the definition that $\mu_{S}$ agrees with $\mu$ almost everywhere on $S$, and, moreover, the coefficient $\mu_S$ is $G_S$-invariant, i.e., almost everywhere in $\hC$
$$
g^*(\mu_S)=\mu_S,\quad {\rm for\ all}\ g\in G_S.
$$ 
Applying the Measurable Riemann Mapping Theorem we obtain a quasiconformal map $h_S$ with 
$$
\mu_{h_S}=\mu_S.
$$ 

The following lemmas are needed in the proof of Theorem~\ref{T:surf}.
\begin{lemma}\cite[Lemma~7.2]{BKM09}\label{L:EqvExt}
If $S$ is a Schottky set and $h$ is a quasiconformal homeomorphism of $\hC$ with a $G_S$-invariant Beltrami coefficient,  then $S'=h(S)$ is a Schottky set. 
\end{lemma}
The proof is straightforward and uses the following observation. The $G_S$-invariance of $\mu_h$ implies that an orientation reversing homeomorphism of $\hC$ that is  conjugate by $h$ to the reflection in a peripheral circle of $S$ must preserve the standard complex structure of $\hC$, and thus is a M\"obius transformation. Since the fixed set of an orientation reversing M\"obius transformation is a circle in $\hC$, the claim follows.  

A related result is the following.
\begin{lemma}\cite[Proposition~5.5]{BKM09}\label{L:EqvExt2}
Let $f\: S\to S$ be a quasisymmetric map of a Schottky set $S$. 
Then $f$ has a $G_S$-equivariant quasiconformal extension $f_S$ to $\hC$. Namely, there exists a quasiconformal homeomorphism  $f_S$ of $\hC$ such that the restriction $f_S|_S$ coincides with $f$, and  for each $g\in G_S$ there exists $g'\in G_S$ with $f_S\circ  g=g'\circ f_S$ in $\hC$.
\end{lemma}

%For reader's convenience, we outline the proof of this lemma and refer to~\cite{BKM09} for more details. 
%We let $\dee B_{i_0},\ i_0\in I$, be  a peripheral circle of $S$ with the largest radius, and let $\dee B_{i_0}'$ be the peripheral circle of $S$ that corresponds to $\dee B_{i_0}$ under the given quasisymmetry $f$.
%Let $g_0\in G$ be the reflection in $\dee B_{i_0}$, and $g_0'\in G$ be the corresponding reflection in $\dee B_{i_0}'$. We use $g_0$ and $g_0'$ to double the Schottky sets $S$ across $\dee B_{i_0}$ and $\dee B_{i_0}'$, respectively, and denote the doubled spaces by $S_1=S\cup g_0(S)$ and $S_1'=S\cup g_0'(S)$. Both, $S_1$ and $S_1'$ are Schottky sets since $g_i$ and $g_i'$ are M\"obius transformations. We next extend the map $f$ to $S_1\setminus S$ by the formula 
%\begin{equation}\label{E:Ext}
%f=g_0'\circ f\circ g_0^{-1}. 
%\end{equation}
%We continue this doubling process  indefinitely. Namely, let $B_{i_1},\ i_1\in I$, be a peripheral circle of $S_1$ with the largest radius, $B_{i_1}',\ i_1\in I$, the corresponding peripheral circle of $S_1'$, the map $g_1\in G$ be the reflection in $B_{i_1}$ and $g_1'$ the reflection in $B_{i_1}'$.  We further extend the previously extended map by Equation~\eqref{E:Ext} with $g_0$ and $g_0'$ replaced by $g_1$ and $g_1'$, respectively, etc. In this way we obtain a sequence of peripheral circles $B_{i_k}, B_{i_k}',\ k\in \N\cup\{0\}$, reflections $g_k, g_k',\ k\in \N\cup\{0\}$, and the sequence of extensions of $f$, each still denoted by $f$.    

The idea of the proof is to use reflections to inductively double across peripheral circles to produce a $G_S$-equivariant extension $f_S$ of $f$ defined on the dissipative part of $G_S$, a dense subset of $\hC$. 
Such a map $f_S$ then extends to a homeomorphism of $\hC$. The fact that $f_S$ is quasiconformal follows from application of the Ahlfors--Beurling extension as well as compactness arguments for families of normalized quasiconformal maps with uniformly bounded dilatation.

%It remains to show that the resulting map $f_S$ is quasiconformal. We do this by  first arguing that the given quasisymmetry $f\: S\to S$ has a $K$-quasiconformal extension to all of $\hC$. This is an application of the classical Ahlfors--Beurling extension result. 
%We use a successive application of~\eqref{E:Ext} to redefine a given quasiconformal extension across each $\dee B_{i_k},\ k\in \N\cup\{0\}$. Since the maps $g_k, g_k'$ are M\"obius, the dilatation is unchanged. Finally, we use standard compactness arguments for families of normalized quasiconformal maps with uniformly bounded dilatation.

The following lemma, also needed in the proof of Theorem~\ref{T:surf} below, is contained in~\cite{BKM09} and is elementary.
\begin{lemma}\cite[Lemma~6.1]{BKM09}\label{L:NestedBalls}
Suppose that $f$ is a continuous map of $\R^n,\ n\in\N$, to itself that is differentiable at the origin $0$. Suppose further that there is a sequence of open geometric balls $(B_k)_{k\in\N}$ that contain 0, such that ${\rm diam}(B_k)\to0,\ k\to\infty$, and for each $k\in\N$, the set  $f(B_k)$ is a geometric ball. Then the derivative $Df(0)$ is a (possibly degenerate or orientation reversing)  conformal linear map, i.e., $Df(0)=\lambda T$, where $\lambda\geq 0$ and $T$ is a linear isometry.  
\end{lemma}

\section{Uniformly bounded dilatation}\label{S:UBD}

\noindent
In this and the following sections, unless explicitly stated, $\mathcal M$ is assumed to be a closed Riemann surface equipped with the canonical metric induced from the standard conformal metric of the universal cover. We denote by $\psi\:\widetilde{\mathcal M}\to\mathcal M$ the covering map. This is a local isometry, and hence a conformal map. 

%For the first part of the following lemma cf.~\cite[Lemma~2.3]{KM10}.
\begin{lemma}\cite[Lemma~2.3]{KM10}\label{L:UQC}
Let $f\in{\rm Diff}^1(\mathcal M)$ permute a dense collection $\{D_i\}_{i\in I}$ of domains with bounded geometry, and let $\Lambda=\mathcal M\setminus\cup_{i\in I}D_i$ be the associated minimal set. 
Then the map $f$ as well as all its forward and backward iterates are uniformly quasiconformal along  $\Lambda$. 
\end{lemma}
\begin{proof}
We assume that the minimal set $\Lambda$  has positive measure, for otherwise there is nothing to prove. Also, since the covering map $\psi$ is a local isometry, below we make no distinction between  small disks in $\mathcal M$ and in $\widetilde{\mathcal M}$ and, whenever convenient, identify the map $f$ with its lift $\tilde f$  to $\widetilde{\mathcal M}$ under the covering map  $\psi$. 
 
Let $k\in\Z$ be fixed.
Since $\Lambda$ is nowhere dense and the closures $\overline{D_i},\ i\in I$, are pairwise disjoint,   for each $p\in\Lambda$ there exists a sequence of complementary domains $(D_i)_{i\in\N}$ of $\Lambda$ that accumulate at $p$, i.e., 
$$
{\rm diam}\, D_i\to0\quad {\rm and}\quad d_{\rm Hausd}(\overline{D_i},\{p\})\to0\  {\rm as}\  i\to\infty,
$$  
where $d_{\rm Hausd}$ denotes the Hausdorff distance.
Also, since $f$ is $\mathcal C^1$-dif\-fe\-ren\-ti\-able on a compact surface $\mathcal M$,  we have
\begin{equation}\label{E:Diff}
f^k(q+x)-f^k(q)-Df^k(q)(x)=o(x),\quad x\to0,
\end{equation}
where $|o(x)|/|x|\to0$ as $|x|\to0$, uniformly in $q$. 

The domains $D_i,\ i\in I$, having bounded geometry means that the\-re exist $C\geq1, p_i\in D_i$, and $0<r_i\leq R_i,\ i\in I$, with
$$
B(p_i, r_i)\subseteq D_i\subseteq B(p_i, R_i),\quad {\rm and}\quad {R_i}/{r_i}\leq C.
$$
The Hausdorff limit of a subsequence of the rescaled domains 
$$
\left((D_{i_m}-p_{i_m})/r_{i_m}\right)_{m\in\N}
$$ 
is thus a closed set $\Omega$ such that 
\begin{equation}\label{E:C1}
B(0,1)\subseteq\Omega\subseteq \overline{B(0, C)}.
\end{equation}
 Let $D_{i_m}',\ m\in\N$,  denote the image of $D_{i_m},\ m\in\N$, under the map $f^k$. The domains $D_{i_m}',\ m\in\N$, also satisfy
\begin{equation}\label{E:DomainImages}
B(p_{i_m}', r_{i_m}')\subseteq D_{i_m}'\subseteq B(p_{i_m}', R_{i_m}'),\quad {\rm and}\quad {R_{i_m}'}/{r_{i_m}'}\leq C,
\end{equation}
for some $p_{i_m}'\in D_{i_m}',\ 0<r_{i_m}'\leq R_{i_m}'$, all $m\in\N$, and with the same constant $C\ge1$ as above.
After possibly passing to a further subsequence, we may assume that the sequence 
$$
\left((D_{i_m}'-f^k(p_{i_m}))/r_{i_m}\right)_{m\in\N}
$$ 
Hausdorff converges to a closed set $\Omega'$. Here, we may choose the same sequence of scales $(r_{i_m})_{m\in\N}$ as for $(D_{i_m})_{m\in\N}$ because $f^k$ is in ${\rm Diff}^1(\mathcal M)$, and hence is locally bi-Lipschitz. From~\eqref{E:DomainImages} we know that there exists $p'\in\Omega'$ and $r'>0$ such that 
\begin{equation}\label{E:C2}
B(p',r')\subseteq\Omega'\subseteq \overline{B(p',Cr')},
\end{equation}
where, again, the constant $C\ge1$ is the same as above.
Since Equation~\eqref{E:Diff} is uniform in $q$, in particular holds for $q=p_{i_m},\ q+x\in D_{i_m},\ m\in\N$, one concludes that 
\begin{equation}\label{E:C3}
\Omega'=Df^k(p)(\Omega). 
\end{equation}
Combining~\eqref{E:C1}, \eqref{E:C2}, and \eqref{E:C3}, and using the fact that the constant $C$ does not depend on $p\in\Lambda$ or  $k\in\Z$, we conclude that the map $Df^k(p)$ has uniformly bounded dilatation, as claimed.
\end{proof}

We say that domains in $\R^2$ have the \emph{same shape} if they are homothetic to each other, and a domain in $\R^2$ has a \emph{generic shape} if the only elements of $SL(2,\R)$ preserving it are the elements of $SO(2,\R)$.  Note that most domains, including geometric disks and squares, have generic shape. However, non-circular ellipses and non-square rectangles do not have generic shape.

\begin{corollary}\label{C:CtoLUQC}
If $f\in{\rm Diff}^1(\T^2)$ permutes a dense collection of domains $\{D_i\}_{i\in I}$, and the lifts of $D_i,\ i\in I$, to $\R^2$ have the same generic shape, 
then the map $f$ is conformal along the minimal set $\Lambda$. 
\end{corollary}
\begin{proof}
Since the lifts of the domains $D_i,\ i\in I$, have the same shape, it implies that $D_i,\ i\in I$, have bounded geometry.
%The second part of the lemma follows from Equation~\eqref{E:C3} and the assumption that all the lifts have the same generic shape. Indeed, 
Also, a simple limiting argument gives that $\Omega$ and $\Omega'$ as in the proof of Lemma~\ref{L:UQC} are homothetic to each other. 
Moreover, $\Omega, \Omega'$ have generic shape as it is a closed condition. 
Hence, Equation~\eqref{E:C3} with $k=1$ implies that for each $p\in\Lambda$ we must have $Df(p)=\lambda(p) T(p)$, where $\lambda(p)\neq0$ is a scalar and $T(p)$ is an element of $SO(2)$, i.e., $f$ is conformal along $\Lambda$.   
\end{proof}

In the proof of Theorem~\ref{T:surf} we use a modification, Proposition~\ref{P:Tukia} below, of a result by P.~Tukia~\cite{Tu80} (see also~\cite{Su78}). 
Recall, a \emph{Beltrami form} on a Riemann surface $\mathcal M$ is a measurable $(-1,1)$-form $\mu=\mu(z)d\bar z/dz$, where we assume that $\|\mu\|_\infty<1$; see, e.g., \cite{BF14} for background on Beltrami forms. For a measurable subset $X$ of $\mathcal M$, a \emph{Beltrami form on} $X$ is a Beltrami form on $\mathcal M$ that is equal to 0 outside $X$.
If $h\: U\to V$ is a quasiconformal map between open subsets $U$ and $V$ of $\mathcal M$ and $\mu$ is a Beltrami form on $V$ (extended by 0 elsewhere), the \emph{pullback} of $\mu$ under $h$ is the Beltrami form given by
$$
h^*(\mu)=\left(\frac{h_{\bar z}+\mu(h)\overline{h_z}}{h_{z}+\mu(h)\overline{h_{\bar z}}}\right)\frac{d\bar z}{dz}.
$$ 
If $h$ is a quasiconformal map of $\mathcal M$, we say that a Beltrami form $\mu$ is  \emph{invariant} under $h$, or that the map $h$ \emph{preserves} a measurable bounded conformal structure,  if
$$
\mu=h^*(\mu).
$$ 

\begin{proposition}\label{P:Tukia}
Let $\Lambda$ be a measurable subset of a Riemann surface $\mathcal M$. Let $H$ be a countable group of quasiconformal maps $h$ of $\mathcal M$ that leave $\Lambda$ invariant and that are  uniformly quasiconformal along $\Lambda$. Then there exists a Beltrami form $\mu$ on $\Lambda$ that is invariant under each map $h\in H$. 

As a consequence, under the above assumptions and if $\mathcal M=\hC$, there exists a quasiconformal map $\gamma$ of $\hC$ that is conformal in $\hC\setminus \Lambda$ and that conjugates each $h\in H$ to a map conformal along $\gamma(\Lambda)$. Conversely, if $H$ is a group of quasiconformal maps of $\hC$ each element of which leaves a fixed Beltrami form $\mu$ on a measurable set $\Lambda$  invariant, then $H$ consists of uniformly quasiconformal maps along $\Lambda$. 
\end{proposition}
\begin{proof}
As in Lemma~\ref{L:UQC}, we may assume that $\Lambda$ has positive measure for otherwise there is nothing to prove.

In local charts of the universal cover, for arbitrary quasiconformal maps $f, g$, expressing $\mu_{g\circ f}$ in terms of $\mu_g$ and $\mu_f$, we have
\begin{equation}\label{E:Composition}
\mu_{g\circ f}=f^*(\mu_g)=T_{f}(\mu_g(f)),
\end{equation}
where
$$
T_{f}(w)=\frac{\mu_f+\frac{\overline{f_z}}{f_z}w}{1+\overline{\mu_f}\frac{\overline{f_z}}{f_z}w}
$$
is an isometry of the disk model hyperbolic plane $\D$.

We consider
$$
X(p)=\{\mu_{h}(p)\:\ h\in H\},
$$
a subset of $\D$ defined for almost every $p\in  \Lambda$ up to a rotation, which depends on a local chart for $p$ and is an isometry of the hyperbolic plane $\D$.  Since $H$ is a group, Equation~\eqref{E:Composition} gives 
\begin{equation}\label{E:Transform}
\begin{aligned}
T_{h}(X(h(p)))&=\left\{T_h(\mu_{h'}(h(p)))\:h'\in H\right\}\\&=\left\{\mu_{h'\circ h}(p)\:h'\in H\right\}=X(p),
\end{aligned}
\end{equation}
for each $h\in H$, almost every $p$ in $\Lambda$, and an appropriate choice of local charts for $p$ and $h(p)$.
If $X$ is a non-empty bounded subset of $\D$, we define by $P(X)$ to be the center of the smallest closed hyperbolic disk that contains $X$. Elementary hyperbolic geometry  shows that $P(X)$ is well defined.
Moreover, $P$ satisfies the following properties (see~\cite{Tu80} for details): 

\noindent
1) if $X$ is bounded above by $\delta$ in the hyperbolic metric of $\D$, then $P(X)$ is at most hyperbolic distance $\delta$ from the origin in $\D$;

\noindent
2) the map $P$ commutes with each hyperbolic isometry of $\D$, i.e., 
$$P(T(X))=T(P(X))$$ for any hyperbolic isometry $T$ of $\D$ and arbitrary non-empty bounded subset $X$ of $\D$.  

We now define
$$
\mu(p)=
\begin{cases}
P(X(p))\frac{d\bar z}{dz},\quad {\rm for\ almost\ every}\ p\in \Lambda,\\
0,\quad {\rm elsewhere\ in}\ \mathcal M. 
\end{cases}
$$
The assumption that the elements of $H$ are $K$-quasiconformal along $\Lambda$ with the same constant $K$ and property 1) above give that $\mu$ is an essentially bounded Beltrami form.
Equation~\eqref{E:Transform} combined with the commutative  property 2) give, in appropriate local charts,
$$
\mu(p)=P(T_{h}(X(h(p))))=T_{h}(P(X(h(p))))=T_{h}(\mu(h(p))),
$$
for each $h\in H$ and almost every $p$ in $\Lambda$. 
This is equivalent to 
\begin{equation}\label{E:invbc}
\mu=h^*(\mu),
\end{equation}
on $\Lambda$ for each $h\in H$, i.e., $\mu$ is invariant under each $h\in H$.

Now, if $\mathcal M=\hC$, by the Measurable Riemann Mapping Theorem  there exists a quasiconformal map $\gamma$ of $\hC$ with $\mu_\gamma=\mu$. Since $\mu=0$ outside $\Lambda$, the map $\gamma$ is conformal in $\hC\setminus \Lambda$. It remains to check that
$$
\mu_{\gamma\circ h\circ \gamma^{-1}}(p)=0
$$
for each $h\in H$  and at almost every point $p$ of  $\gamma(\Lambda)$. This is equivalent to 
\begin{equation}\label{E:Conjugacy}
\mu_\gamma(p)=\mu_{\gamma\circ h}(p),
\end{equation}
for each $h\in H$  and at almost every point $p$ of  $\Lambda$. But $\mu_{\gamma\circ h}=h^*(\mu_\gamma)$, and since $\mu_\gamma=\mu$, Equation~\eqref{E:Conjugacy} is the same as~\eqref{E:invbc}, i.e., each map $\gamma\circ h\circ \gamma^{-1},\ h\in H$, is conformal along $\gamma(\Lambda)$.

Finally, to prove the converse statement, we just observe that, if $\gamma$ is a quasiconformal map of $\hC$ with $\mu_\gamma=\mu$, each $\gamma\circ h\circ \gamma^{-1},\ h\in H$, is conformal along $\gamma(\Lambda)$, and hence the elements $h\in H$ are uniformly quasiconformal along $\Lambda$. 
\end{proof}

An alternative proof of Proposition~\ref{P:Tukia} uses the barycenter of the convex hull to define $\mu$ rather than the center of the smallest hyperbolic disk as above. This approach was used in~\cite{Su78}.

\section{Covering properties}\label{S:CoPr}

\noindent
In this section we prove auxilliary lemmas that will be used in the proof of Theorem~\ref{T:surf2}. The reader may initially want to skip the proofs in this section and move to the next one where Theorem~\ref{T:surf2} is proved.

\begin{lemma}\label{L:UQ} 
Let $\mathcal M$ be a closed Riemann surface that is either hyperbolic or flat 2-torus $\T^2$.
Let $\{D_i\}_{i\in I}$ be a collection of uniform quasidisks in $\mathcal M$. 
If $\{\widetilde D_i\}_{i\in \tilde I}$ is the collection of lifts of $D_i,\ i\in I$, to the universal cover $\widetilde{\mathcal M}$, then $\{\widetilde D_i\}_{i\in \tilde I}$ consists of uniform quasidisks in the spherical  metric of $\D$ or $\C$,  respectively. 
\end{lemma}
\begin{proof}
First, the lifts of the closures $\overline{D_i},\ i\in I$, to the universal cover are closed topological disks. Indeed, this follows from the facts that each $\overline{D_i}$ is embedded in $\mathcal M$ and that the covering map $\psi$ is a local homeomorphism. 

We use the two point condition~\eqref{E:BT} for quasicircles to  verify that the lifts $\widetilde D_i,\ i\in\tilde I$,  of $D_i,\ i\in I$, are uniform quasidisks in the intrinsic metric of the universal cover $\widetilde{\mathcal M}$, which is either $\D$ with the hyperbolic metric or $\C$ with the Euclidean metric. First, since $\mathcal M$ is compact and $D_i,\ i\in I$, are uniform quasidisks, we must have ${\rm diam}(D_i)\to0$, i.e., given any $\delta>0$, only finitely many of $D_i,\ i\in I$, have diameters at least $\delta$. For, suppose not. Then there exists a sequence $(i_k)_{k\in\N}$, of distinct elements such that 
 ${\rm diam}(D_{i_k})\ge\delta$ for some $\delta>0$. Let $p^1_{k}$, $p^2_k\in\dee D_{i_k}$ be two points with
 $$
 {\rm dist}(p^1_{k}, p^2_k)={\rm diam}(D_{i_k}),\quad {\rm for\ all}\ k\in\N.
  $$  
  By passing to subsequences, we may assume that $(p^1_{k})_{k\in\N}$,  $(p^2_{k})_{k\in\N}$ converge to points $p^1, p^2$, respectively, necessarily with ${\rm dist}(p^1, p^2)\ge\delta$. Now, if $0<\epsilon<\delta/2$ is such that in the disk $B(p^1,2\epsilon)$ an inverse branch of the covering map $\psi$ is an isometry, we get a contradiction with the two point condition for $k$ large enough, for pairs of points in the intersection of $\dee D_{i_k},\ k\in\N$, with the boundary circle of $B(p^1,\epsilon)$.
  
Next, the intrinsic diameters of $\widetilde D_i,\ i\in\tilde I$, are uniformly bounded above. Indeed, this follows from the facts that each $\widetilde D_i,\ i\in\tilde I$, projects to some $D_i,\ i\in I$, the diameters of the latter go to 0, and the covering map $\psi$ is a local isometry.

To show that $\widetilde D_i,\ i\in\tilde I$, are uniform quasidisks,  assume that there exists a (possibly constant) sequence $(i_k)_{k\in\N}$ and two sequences of distinct points $(\tilde p^1_k)_{k\in\N}, (\tilde p^2_k)_{k\in\N}$, with $\tilde p^1_k, \tilde p^2_k\in\partial \widetilde D_{i_k}$, such that 
\begin{equation}\label{E:twoptfails}
{\rm dist}(\tilde p^1_k, \tilde p^2_k)/\min\{{\rm diam}(\tilde J^1_k),{\rm diam}(\tilde J^2_k)\}\to0,\quad k\to\infty,
\end{equation}
where $\tilde J^1_k, \tilde J^2_k$ are two complimentary arcs of $\partial \widetilde D_{i_k}\setminus\{\tilde p^1_k,\tilde p^2_k\}$. Here, the distances and diameters are in the intrinsic metric of the universal cover $\widetilde{\mathcal M}$.
Since ${\rm diam}(\tilde J^1_k),{\rm diam}(\tilde J^2_k)$ are uniformly bounded above and  $\mathcal M$ is compact, we must have ${\rm dist}(\tilde p^1_k, \tilde p^2_k)\to0,\ k\to\infty$, and for $p^1_k=\psi(\tilde p^1_k), p^2_k=\psi(\tilde p^2_k),\ k\in\N$,  by possibly passing to subsequences, we may assume that $p^1_k,p^2_k\to p,\ k\to\infty$. 
Using the fact that the covering map is a  local isometry, let $\Omega$ and $\widetilde\Omega$ be open disks in $\mathcal M$ and  $\widetilde{\mathcal M}$, respectively, such that $\Omega$ contains $p$ and $\psi\: \widetilde\Omega\to\Omega$ is an isometry. 
By applying deck transformations, i.e., by acting by elements of the fundamental group, which are isometries, we may assume that for all $k$ large enough $\tilde p^1_k, \tilde p^2_k \in\widetilde\Omega$. If $J_k^1=\psi(\tilde J_k^1), J_k^2=\psi(\tilde J_k^2)$, the uniform two point condition for $\{D_i\}_{i\in I}$ gives $\min\{{\rm diam}(J^1_k),{\rm diam}(J^2_k)\}\to0,\ k\to\infty$. But that implies that for $k$ large enough, one of the arcs $J^1_k, J^2_k$ with the smaller diameter is contained in $\Omega$.
 Convergence~\eqref{E:twoptfails} thus descends under $\psi$ to $\mathcal M$, namely it holds for $D_{i_k}, p^1_k, p^2_k$ in place of $\widetilde D_{i_k}, \tilde p^1_k, \tilde p^2_k$, respectively. 
The domains $D_i,\ i\in I$, are assumed to be uniform quasidisks however, and thus we arrive at a contradiction and conclude that $\widetilde D_i,\ i\in\tilde I$, are uniform quasidisks in the intrinsic metric of the universal cover $\widetilde{\mathcal M}$. 

We next argue that $\widetilde D_i,\ i\in\tilde I$, are uniform quasidisks in the spherical metric.
Let $\mathcal F$ be a fixed fundamental domain in $\widetilde{\mathcal M}$ under the action of the fundamental group of $\mathcal M$.
Since $\mathcal M$ is compact, we may assume that the closure $\overline{\mathcal F}$ of $\mathcal F$ is compact in $\widetilde{\mathcal M}$. It was proved above that all $\widetilde D_i,\ i\in\tilde I$, have uniformly bounded diameters. Since in any compact subset of $\widetilde{\mathcal M}$ the intrinsic and spherical metrics are bi-Lipschitz equivalent, all $\widetilde D_i$ that intersect $\overline{\mathcal F}$ are uniform quasidisks in the spherical metric. 
For an arbitrary $\widetilde D_i$, i.e., one that does not intersect $\overline{\mathcal F}$, we can use the Koebe Distortion Theorem (see, e.g., \cite[p.~32]{Du83}, or the Egg Yolk Principle~\cite[Theorem~11.14]{He01}), because deck transformations are conformal automorphisms of $\widetilde{\mathcal M}$. 
%Namely, an arbitrary $\widetilde D_i$ is the image under a deck transformation of a domain $\widetilde D_j$ that intersects $\overline{\mathcal F}$. Hence all $\widetilde D_i,\ i\in\tilde I$, are uniform quasidisks in the spherical metric. 
\end{proof}

\begin{lemma}\label{L:US} 
Let $\mathcal M$ be a closed Riemann surface that is either hyperbolic or flat 2-torus $\T^2$.
Let $\{D_i\}_{i\in I}$ be a collection of domains in $\mathcal M$ that consists of uniform quasidisks that are uniformly relatively separated. 
If $\{\widetilde D_i\}_{i\in \tilde I}$ is the collection of lifts of $D_i,\ i\in I$, to the universal cover, then $\widetilde D_i,\ i\in \tilde I$, are uniformly relatively separated in the spherical  metric of $\D$ or $\C$,  respectively. Also, in the case when $\mathcal M$ is hyperbolic,  $\T, \widetilde D_i,\ i\in \tilde I$, are uniformly relatively separated in the spherical metric of $\overline \D$.
\end{lemma}
\begin{proof}
%Now we show that $\widetilde D_i,\ i\in\tilde I$, are uniformly relatively separated in the Euclidean metric. 
The argument is similar to establishing that all $\widetilde D_i,\ i\in\tilde I$, are uniform quasidisks, and we first show that uniform relative separation holds in the intrinsic metric. Namely,  assume that for two sequences $(\widetilde D_{i_k})_{k\in\N}, (\widetilde D_{j_k})_{k\in\N}$ of domains we have
$$
\frac{{\rm dist}(\widetilde D_{i_k},\widetilde D_{j_k})}{\min\{{\rm diam}(\widetilde D_{i_k}),{\rm diam}(\widetilde D_{j_k})\}}\to0,\quad k\to\infty.
$$
By applying elements of the fundamental group of $\mathcal M$, we may assume that one of the domains, say $\widetilde D_{i_k}$, intersects the closure $\overline{\mathcal F}$ of a fixed fundamental domain. 
Since, by the proof of Lemma~\ref{L:UQ} the intrinsic diameters of $\widetilde D_i,\ i\in\tilde I$, are uniformly bounded above, 
 we must have ${\rm dist}(\widetilde D_{i_k},\widetilde D_{j_k})\to0,\ k\to\infty$.
Also, since the fundamental group of $\mathcal M$ is finitely generated, we may assume that for the projections $D_{i_k}=\psi(\widetilde D_{i_k}), D_{j_k}=\psi(\widetilde D_{j_k})$ under the covering map one has $D_{i_k}\neq D_{j_k}$ for all $k$, and ${\rm dist}(D_{i_k},D_{j_k})\to0,\ k\to\infty$. Hence, the uniform relative se\-pa\-ra\-tion assumption on $D_i,\ i\in I$, gives $\min\{{\rm diam}(D_{i_k}),{\rm diam}(D_{j_k})\}\to0,\ k\to\infty$. Compactness of $\mathcal M$  gives a point $p\in\mathcal M$ such that, after possibly passing to subsequences, one of the sequences, say $(D_{i_k})$, converges to $p$. Local isometry property of the covering map $\psi$ near $p$ along with the assumption of uniform relative separation of $D_i,\ i\in I$, produce a contradiction, and thus $\widetilde D_i,\ i\in\tilde I$, are uniformly relatively separated in the intrinsic metric. Now, in any given compact subset of $\D$, the intrinsic and spherical metrics are bi-Lipschitz equivalent to each other, and thus we conclude uniform relative separation in the spherical metric for domains $\widetilde D_i,\ i\in\tilde I$, intersecting $\overline{\mathcal F}$. Again, one uses the Koebe Distortion Theorem  to conclude that uniform relative separation  in the Euclidean metric spreads around in $\widetilde{\mathcal M}$.  

The assertion that, in the hyperbolic surface case, $\widetilde D_i,\ i\in\tilde I$, are uniformly relatively separated in the spherical metric from the unit circle boundary $\T$ follows from  the earlier conclusions that $\widetilde D_i,\ i\in\tilde I$, have uniformly bounded diameters in the hyperbolic metric and are uniform quasidisks in the spherical metric,  as well as elementary hyperbolic geometry. Namely, the fact that if a hyperbolic radius of a disk $D$ in $\D$ is fixed, its spherical radius is roughly proportional to the spherical distance of $D$ to the boundary circle $\T$.
%Finally, we establish uniform quasidisks and uniform relative separation conditions for $\widetilde D_i,\ i\in\tilde I$, in the spherical metric. In the case of a hyperbolic surface, the Euclidean and the spherical distances in the closed unit disk are bi-Lipschitz equivalent, and so being uniform quasidisks or uniformly relatively separated is the same in both, the Euclidean and spherical, metrics. For the case of the torus $\T^2$, in any disk of radius $R$ centered at the origin, the Euclidean and the spherical metrics are comparable with constants depending on $R$. For domains $\widetilde D_i$ outside of some large disk of radius $R$, we can use the inversion $z\mapsto1/z$ to transplant those domains to the unit disk. The conclusion that such domains are uniformly relatively separated uniform quasidisks in the spherical metric follows from the following observations. First,  the Koebe Distortion Theorem guarantees that the uniformness of quasidisks as well as uniform relative separation are preserved under the inversion in the Euclidean metric; second, in the unit disk the spherical and the Euclidean metrics are comparable; and, finally, the inversion is an isometry in the spherical metric.  
\end{proof}

\section{Proofs of Theorems~\ref{T:surf} and~\ref{T:surf2}}\label{S:Surf}

\noindent
We first reduce Theorem~\ref{T:surf} to Theorem~\ref{T:surf2}.
By Lemma~\ref{L:UQC}, a map $f\in{\rm Diff}^1(\mathcal M)$ and all its  forward and backward iterates have uniformly bounded quasiconformal dilatation almost everywhere on $\Lambda$. By the first part of Proposition~\ref{P:Tukia} applied to the cyclic group  generated by the map $f$ we conclude that there exists a Beltrami form on $\Lambda$ invariant under $f$. Since each $f\in{\rm Diff}^1(\mathcal M)$ on a closed Riemann surface $\mathcal M$ is quasiconformal, Theorem~\ref{T:surf} is reduced to Theorem~\ref{T:surf2}.
We proceed by proving Theorem~\ref{T:surf2}.

We assume that $f$ as in Theorem~\ref{T:surf2} exists and let $\tilde f$ denote a lift of $f$ to the universal cover $\widetilde{\mathcal M}$ of $\mathcal M$. Below, we identify $\widetilde{\mathcal M}$ with either $\D$ or $\C$, depending on whether $\mathcal M$ is hyperbolic or $\T^2$. Let $D$ denote the group of deck transformations, which is finitely generated. Finally, let $\widetilde\Lambda$ be the preimage of $\Lambda$ under the covering map $\psi\colon\widetilde{\mathcal M}\to\mathcal M$. Also, we denote by $H$ the group acting on $\widetilde{\mathcal M}$ generated by $\tilde f$ and generators of $D$. The group $H$ is finitely generated and $\tilde f$ commutes with elements of the group $D$ because $\tilde f$ is a lift of a map $f\:\mathcal M\to\mathcal M$. 

The first step is to apply Theorem~\ref{T:Unif} to $\widetilde\Lambda$ viewed as a subset of $\hC$ via the stereographic projection.  By Lemmas~\ref{L:UQ} and~\ref{L:US} conditions of that theorem are met.
Thus there exists a quasiconformal map $\beta$ of $\hC$ such that each $\widetilde D_i,\ i\in \tilde I$, is mapped  by $\beta$ to a geometric disk $\widetilde B_i$. We may assume that, in the hyperbolic surface case, the unit disk is preserved by $\beta$.
 The image $S$ of the closure of $\widetilde\Lambda$ in $\hC$ under $\beta$ is thus a {Schottky set} in $\hC$, i.e., a subset of $\hC$ whose complement is a union of disjoint geometric disks. 
Let $G$ be the {Schottky group} associated to $S$, i.e., the group generated by reflections in $\dee\widetilde B_i,\ i\in \tilde I$, and,  in the hyperbolic surface case, by the reflection in $\T$. 
 We denote by $H_\beta$ the conjugate group of $H$ by the map $\beta$.  Recall that $H$ is generated by the group $D$ of deck transformations as well as the map $\tilde f$ acting on $\widetilde M$.

From the converse statement of Proposition~\ref{P:Tukia} we know that the map $f$ and all its iterates are uniformly quasiconformal along $\Lambda$.
The covering map being  a local isometry implies that the map $\tilde f$ along with all its iterates have uniformly bounded quasiconformal dilatation almost everywhere on $\widetilde\Lambda$.   
Since $\beta$ is quasiconformal, the conjugate $\tilde f_\beta$ of $\tilde f$ by $\beta$ and all its iterates, forward and backward, are uniformly quasiconformal  along $S$. Likewise, since the elements of $D$ are conformal and commute with $\tilde f$, all elements of  $H_\beta$  are uniformly quasiconformal along $S$.

Now, we apply Proposition~\ref{P:Tukia} to the Schottky set $S$ and the group $H_\beta$. This way we obtain a quasiconformal map $\gamma$ of $\hC$, conformal at almost every point in $\hC\setminus S$, that conjugates each map $h_\beta\in H_\beta$ to a map $h_\gamma$  that is conformal along $\gamma(S)$.  
Let $\mu_\gamma$ be the Beltrami coefficient of $\gamma$. As discussed in Section~\ref{S:SSG}, we can redefine  $\mu_\gamma$ on the dissipative part of the Schottky group $G$, other than on $S$, to obtain a $G$-invariant Beltrami coefficient $\mu_G$. Let $\gamma_G$ be a quasiconformal map with Beltrami coefficient $\mu_G$, which exists by the Measurable Riemann Mapping Theorem. By Lemma~\ref{L:EqvExt}, the map $\gamma_G$ must take the Schottky set $S$ onto another Schottky set $S'$. We continue to denote by $G$ the Schottky group generated by reflections in the peripheral circles of $S'$. Indeed, the Schottky groups associated to $S$ and $S'$ are isomorphic as abstract groups. 
We denote by $H_G$ the quasiconformal group conjugate to $H_\beta$ by the quasiconformal map $\gamma_G$.
Because $\mu_G=\mu_\gamma$ on $S$, each map  $h\in H_G$ is conformal along the Schottky set $S'$. 

Using the fact that a quasiconformal map of $\hC$ is quasisymmetric and applying Lemma~\ref{L:EqvExt2}, we conclude that the restriction of each $h\in H_G$ to $S'$ has a $G$-equivariant quasiconformal extension $h_G$ to $\hC$. Equivariance of $h_G$, in turn, guarantees that if $\dee B_i,\ i\in \tilde I$, is an arbitrary peripheral circle of $S'$ and $g\in G$ is arbitrary, then $h_G(g(\dee B_i))$ is a geometric circle  in $\hC$. Indeed, the equivariance of $h_G$ gives
$$
h_G(g(\dee B_i))=g'(h_G(\dee B_i))=g'(h(\dee B_i)),
$$
for some $g'\in G$. Since $h(\dee B_i)$ is a peripheral circle of $S'$, and $g'$ is a M\"obius transformation, the claim follows.

\begin{lemma}\label{L:hConf}
Each $h_G$ is a M\"obius transformation.
\end{lemma}
\begin{proof} 
Since $h_G$ is quasiconformal, it is  differentiable almost everywhere in $\C$. 
Thus it is enough  to show that $h_G$ is conformal at almost every point $p$ of differentiability.
We consider the following two cases: either $p$ belongs to the dissipative part of $G$, i.e., $p\in g(S')$ for some $g\in G$, or $p$ is in the conservative part, i.e., there exists a sequence $(g_k)_{k\in\N}$ of elements of $G$ such that $p\in B_{i_k},\ k\in\N$, where $B_{i_k}$  is a complementary disk component of $g_k(S')$. In the latter case, the density of the dissipative part of $G$ gives $\{p\}=\cap_{k\in\N} B_{i_k}$. The first case is only relevant when $S'$, and hence the dissipative part, has positive measure, while the second case needs to be considered only when the conservative part has positive measure.

We start with the first case, i.e., we assume that $S'$ has positive measure and $p\in g(S')$ for some $g\in G$. We may further assume that $p$ is a Lebesgue density point of $g(S')$ because such points form a set of full measure. But, since $g$ is a M\"obius transformation and $h_G$ is $G$-equivariant and conformal along $S'$, the map $h_G$ is conformal at almost every such point $p$ in the dissipative part of $G$.  

We now deal with the conservative part, i.e., when 
$$
\{p\}=\cap_{k\in\N} B_{i_k},\ {\rm diam}(B_{i_k})\to0,
$$ 
as $k\to\infty$, where $B_{i_k},\ k\in\N$,  is a complementary disk of $g_k(S')$ with $g_k\in G$. 
Pre- and post-composing $h_G$ with appropriate M\"obius transformations we obtain a map $h_p$ for which we can assume that  $p=h_p(p)=0$. 
Because $h_G$ is $G$-equivariant, for all $k\in\N$, it takes complementary disks of $g_k(S')$  to complementary disks of $g_k'(S')$, for some $g_k'\in G,\ k\in\N$. Since all the complementary disks are geometric disks and translations preserve such disks, the conformality of $h_p$ at 0 follows from Lemma~\ref{L:NestedBalls}. Hence $h_G$ is conformal at $p$.
The map  $h_G$ being quasiconformal in $\hC$ and conformal at almost every point implies that it is conformal in $\hC$ by Weyl's lemma.
%; cf.~\cite[Lemma~3.2]{Me19}. 
Thus $h_G$ must be a M\"obius transformation. 
\end{proof}
 
 We next replace $h_G$ by a M\"obius transformation $h_G'$ that fixes 0 and $\infty$. In the case of 2-torus $\T^2$, the map $h_G$ itself fixes $\infty$. Post-composing $h_G$ with a translation, we can guarantee that 0 is also fixed. In the  hyperbolic surface case we may post-compose $h_G$ by an appropriate M\"obius transformation preserving the unit disk so that 0 and $\infty$ become fixed. This is possible because $h_G$ is $G$-equivariant, and, particularly, commutes with the reflection in $\T$. 

\begin{lemma}\label{L:hIso}
Each $h_G'$ is an isometry of $\C$.
\end{lemma}
\begin{proof}
Because $h_G'$ fixes 0 and $\infty$, it has to be of the form $h_G'(z)=\lambda Tz$ in  $\C$, where $\lambda>0$ and $T$ is a linear isometry. In addition,  since in the hyperbolic surface case the map $h_G'$ preserves the unit disk, we must have that $\lambda=1$, i.e., $h_G'$ is a rotation. The same is true in 2-torus $\T^2$ case.
Indeed, from the arguments above we know that the deck transformations $d_1(z)=z+1$ and $d_i(z)=z+i$ for $\T^2$ are conjugated on $\widetilde\Lambda$ by the quasiconformal map $\gamma_G\circ\beta$ to the linear maps $\delta_1(z)=a_1z+b_1,\ \delta_i(z)=a_2z+b_2$ of $S'$, respectively. Assume that $a_1\neq 1$. Then $\delta_1$ has a fixed point $p$ in $\C$. It cannot be in $S'$ because $d_1$ does not have fixed points. Then the fixed point $p$ must be in one of the complementary disks of $S'$. This disk must then be fixed setwise by $\delta_1$, and hence $|a_1|=1$, i.e., $\delta_1$ is a Euclidean isometry. Therefore, we can conjugate $\delta_1$ by a translation to a rotation of $\C$ that preserves a Shottky set. Such a rotation would have to have a finite order, a contradiction because $d_1$ has infinite order. Thus $a_1=1$ and $\delta_1$ is a translation of $\C$. 
Likewise, the map $\delta_2$ is a translation, and, in particular both maps, $\delta_1$ and $\delta_2$, are isometries of the Euclidean plane. It follows then that the Euclidean radii of the complementary disks of $S'$ are uniformly bounded in the $\T^2$ case. This, in turn, implies that the restriction to $S'$ of the conjugate $f_G$ of the map $\tilde f$ by $\gamma_G\circ\beta$  is  a Euclidean isometry. 
Thus, in either case, we have $h_G'(z)=Tz$ is a linear isometry of $\C$. 
\end{proof}

Lemmas~\ref{L:hConf} and~\ref{L:hIso} imply that $h_G$ is an isometry in the intrinsic metric of $\widetilde{\mathcal M}$ because $h_G$ and $h_G'$ differ by a post-composition with an intrinsic isometry of $\widetilde{\mathcal M}$. In particular, each map $h_G\in H_G$, cannot change the intrinsic radii  of complementary  disks of $S'$. 
This is impossible. Indeed, the map $f$ is assumed to permute complementary domains of $\Lambda$. Let $\Omega$ be a fixed complementary component of $ \Lambda$, let $\widetilde\Omega$ be a fixed lift of $\Omega$ by the covering map $\psi$, and let $B$ be the disk $\gamma_G\circ\beta(\widetilde\Omega)$.  For the $G$-equivariant map $f_G$, the sequence $(B_k)_{k\in\N}$, where $B_k=f_G^k(B),\ k\in\N$, consists of the disks with the same intrinsic radii. We can apply conjugates of the deck transformations, which are isometries, to bring each disk $B_k$ to intersect a fixed fundamental domain. However, there are only finitely many disks of the same radius that can intersect such a fundamental domain. This implies that, up to an application of deck transformation, $B_{k_1}=B_{k_2}$ for some $k_1, k_2\in\N, k_1\neq k_2$. This, in turn, implies that $\tilde f^{k_1}(\widetilde\Omega)$ differs from  $\tilde f^{k_2}(\widetilde\Omega)$ by an element of a deck transformation, i.e., $f^{k_1}(\Omega)=f^{k_2}(\Omega)$ for $k_1\neq k_2$, a contradiction.
\qed

\section{A non-example}\label{S:Ex}

\noindent
To illustrate the subtlety of ruling out existence of wandering domains, we give a non-example that is not covered by either Theorem~\ref{T:tori} or Theorem~\ref{T:surf}. 
\begin{proposition}\label{P:Ex}
A Denjoy type $f\in{\rm Diff}^1(\T^2)$ with square wandering domains  does not exist.
\end{proposition}
 Here, by \emph{square} wandering domains in 2-torus $\T^2=\R^2/\Z^2$ we mean domains in $\T^2$ whose lifts to $\R^2$ are non-trivial geometric squares with sides parallel to the coordinate axes; see a cartoon Figure~\ref{F:TS}.
\begin{figure}
[htbp]
\begin{center}
\includegraphics[height=40mm]{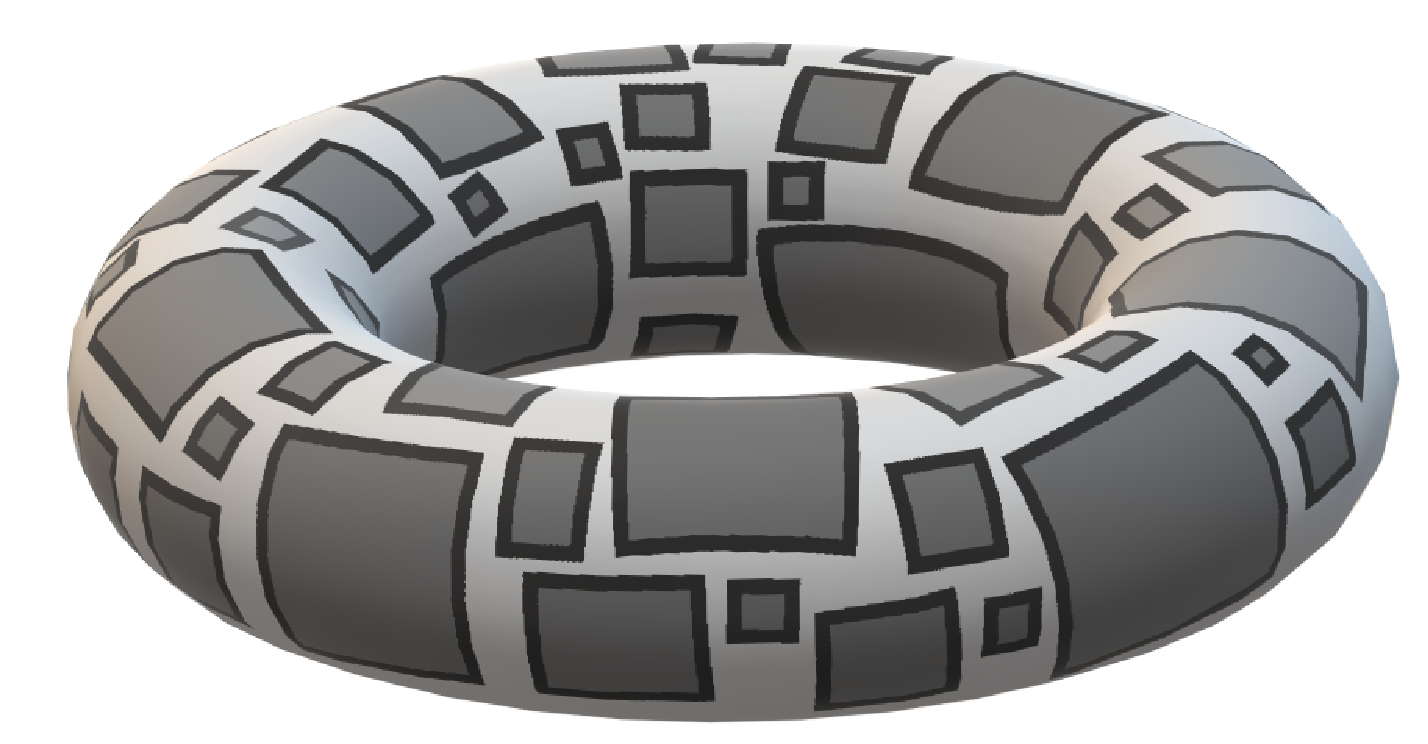}
\caption{
Square wandering domains in $\T^2$.
}
\label{F:TS}
\end{center}
\end{figure}
To rule such a diffeomorphism out we cannot use Theorem~\ref{T:tori} because we do not assume that the minimal set has measure zero, and we cannot use  Theorem~\ref{T:surf} because  wandering domains are not assumed to be  uniformly relatively separated.

\begin{proof}[Proof of Proposition~\ref{P:Ex}.]
Fir\-st, if exists, a map $f$ is conformal along its minimal set $\Lambda$ by Corollary~\ref{C:CtoLUQC}. Next, as in the proof of Theorem~\ref{T:tori}, Section~\ref{S:Tori}, we consider the horizontal curve family $\Gamma_h$ in $\T^2$, and argue that the extremal distribution $\rho_h$ for its transboundary modulus  is the one that assigns constant value 1 to $\Lambda$ and to each complementary square domain its side length. Indeed, trivially, $\rho_h$ is an admissible mass distribution whose total mass equals 1. Let $\rho$ be an arbitrary admissible mass distribution for $\Gamma_h$.
We consider the subfamily $\Gamma_{sh}$ of $\Gamma_h$ that consists of \emph{straight} horizontal paths  $\{\gamma_y(t)=t+iy\: 0\le t\le1\}$, for each $y,\ 0\le y<1$, where $\gamma_y(0)$ and $\gamma_y(1)$ are identified.
Since $\rho$ is extremal for $\Gamma_h$, it is admissible for $\Gamma_{sh}$, and thus
$$
1\le\int_{\gamma_y\cap\Lambda}\rho\, dt+\sum_{\gamma_y\cap D_i\neq\emptyset}\rho(D_i),
$$
for almost every $y\in[0,1)$, where $D_i,\ i\in I$, are the square wandering domains. Integrating this inequality over $y\in[0,1)$, we obtain
$$
1\le\int_\Lambda\rho\, dtdy+\sum_{i\in I}\rho(D_i)s_i,
$$
where $s_i$ is the side length of the square $D_i,\ i\in I$.  
Applying the Cauchy-Schwarz inequality we get
$$
1\le\left(\int_\Lambda\rho^2dtdy+\sum_{i\in I}\rho(D_i)^2\right)^{1/2}\cdot\left(|\Lambda|+\sum_{i\in I}s_i^2\right)^{1/2},
$$
where $|\Lambda|$ denotes the area of $\Lambda$. Since the second factor of the right-hand side equals 1, we conclude that ${\rm mod}_\Lambda(\Gamma_h)\ge1$. But ${\rm mass}(\rho_h)=1$, and thus ${\rm mod}_\Lambda(\Gamma_h)=1$, and the mass distribution $\rho_h$ is extremal.

The map $f$ being isotopic to the identity preserves the family $\Gamma_h$, and, in particular,
$$
{\rm mod}_\Lambda(f(\Gamma_h))={\rm mod}(\Gamma_h).
$$ 
By Lemma~\ref{L:Confinv}, the pullback $f^*(\rho_h)$ is also an extremal mass distribution.
The uniqueness of  the extremal mass distribution claimed in Proposition~\ref{P:Extremaldistr} implies that $\rho_h=f^*(\rho_h)$. In particular, the map $f$ must preserve the side lengths of the squares $D_i,\ i\in I$. The contradiction comes from the fact that one cannot have infinitely many disjoint squares in $\T^2$ of the same side length.
\end{proof}

\end{document}